\documentclass[a4paper,12pt]{amsart}
\usepackage{mathtools}
\usepackage{fullpage}
\usepackage[alphabetic,nobysame, initials]{amsrefs}
\usepackage{enumerate}
\usepackage{hyperref}
\usepackage{color}
\usepackage{bbm, makecell}
\usepackage[pdftex]{graphicx}

\newtheorem{Theorem}{Theorem}[section]
\newtheorem{Lemma}[Theorem]{Lemma}
\newtheorem{Proposition}[Theorem]{Proposition}
\newtheorem{Corollary}[Theorem]{Corollary}
\theoremstyle{definition}
\newtheorem{Definition}[Theorem]{Definition}
\newtheorem{Conjecture}[Theorem]{Conjecture}
\newtheorem{Problem}[Theorem]{Problem}
\newtheorem{Remark}[Theorem]{Remark}
\newtheorem{Question}[Theorem]{Question}

\renewcommand{\Re}{\mathbb R}

\newcommand{\BB}{\mathbf B}
\newcommand{\CC}{\mathbf C}
\newcommand{\KK}{\mathbf K}

\newcommand{\QQ}{\mathbf Q}

\newcommand{\M}{\mathbb{M}}

\newcommand{\xx}{\mathbf x}
\newcommand{\yy}{\mathbf y}
\newcommand{\pp}{\mathbf{p}}

\newcommand{\qqq}{\mathbf q}

\newcommand{\uu}{\mathbf u}

\newcommand{\oo}{\mathbf o}
\newcommand{\Sph}{\mathbb{S}}
\newcommand{\HH}{\mathbb{H}}
\newcommand{\ball}[2][d]{\BB^{#1}\left[#2\right]}
\newcommand{\prob}[1]{{\mathbb P}\left[#1\right]}

\renewcommand{\Re}{{\mathbb R}}

\newcommand{\vect}[1]{{\bf #1}}
\newcommand{\di}{\;{\mathrm d}}
\newcommand{\Ee}{{\mathbb E}}

\newcommand{\Ed}{\Ee^d}
\newcommand{\Ethree}{\Ee^3}

\newcommand{\cl}{{\rm cl}}
\newcommand{\bd}{{\rm bd}}

\newcommand{\FF}{{\mathcal F}}

\newcommand{\noshow}[1]{}
\newcommand{\Sedm}{{\mathbb S}^{d-1}}
\newcommand{\st}{\; : \; }

\newcommand{\conv}{{\rm conv}}
\newcommand{\ivol}[2][k]{{\rm V}_{#1}\left(#2\right)}

\newcommand{\vol}[1]{{\rm V}_d\left(#1\right)}

\newcommand{\body}[1]{\mathbf{#1}}

\DeclareMathOperator{\area}{area}
\DeclareMathOperator{\perim}{per}

\DeclareMathOperator{\surf}{surf}

\newcommand{\LL}{\mathbf{L}}
\newcommand{\OO}{\mathcal{O}}
\DeclareMathOperator{\var}{var}
\DeclareMathOperator{\supp}{supp}

\title{Selected topics from the theory of intersections of balls}

\author{K\'{a}roly Bezdek, Zsolt L\'angi, and M\'arton Nasz\'odi}

\address{
K\'{a}roly Bezdek,
Department of Mathematics and Statistics, University of Calgary, Canada,
and
Department of Mathematics, University of Pannonia, Veszpr\'em, Hungary.
}
\email{bezdek@math.ucalgary.ca}

\address{
Zsolt L\'angi,
Bolyai Institute, University of Szeged, Hungary}
\email{zlangi@server.math.u-szeged.hu}

\address{
M\'arton Nasz\'odi,
Alfr\'ed R\'enyi Institute of Mathematics and
Lor\'and E\"otv\"os University,
Budapest, Hungary.
}
\email{marton.naszodi@renyi.hu}


\begin{document}
\begin{abstract}
 In this survey, we discuss volumetric and combinatorial results concerning (mostly finite) intersections or unions of balls (mostly of equal radii) in the $d$-dimensional real vector space, mostly equipped with the Euclidean norm. Our first topic is the \emph{Kneser--Poulsen Conjecture}, according to which if a finite number of balls are rearranged so that the pairwise distances of the centers increase, then the volume of the union (resp., intersection) increases (resp., decreases).  
 
 Next, we discuss Blaschke--Santal\'o-type inequalities, and reverse isoperimetric inequalities for convex sets in Euclidean $d$-space obtained as intersections of (possibly infinitely many) balls of radius $r$, which we call \emph{$r$-ball bodies}. We present some results on $1$-ball bodies (also called \emph{ball-bodies} or \emph{ spindle convex sets}) in the plane, with special attention paid to their approximation by the spindle convex hull of a finite subset.

 A \emph{ball-polyhedron} is a ball-body obtained as the intersection of finitely many unit balls in Euclidean $d$-space. We consider the combinatorial structure of their faces, and volumetric properties of ball-polyhedra obtained from choosing the centers of the balls randomly.
\end{abstract}
\maketitle

\section{Introduction}
Most of our discussion takes place in Euclidean $d$-space, $\Ed$. Intersections of balls appear in the study of diameter graphs (Borsuk-type problems), sets of constant width or convex bodies with a $C^{2+}$ boundary, and thus, they are studied 
in combinatorial geometry, differential geometry and classical convexity.

The goal of this survey is to collect recent advances in the topic not covered in earlier works cataloging progress, most notably in 
\cites{BLNP, KMP, Be10, Be13, BL19, MMO19}. For further reading on the topics covered in this survey, we recommend \cite{ArFl25} for interested readers.

Our first topic (Section~\ref{sec:KneserPoulsen}) is the \emph{Kneser--Poulsen Conjecture}, according to which if a finite number of unit balls are rearranged so that the pairwise distances of the centers increase, then the volume of the union (resp., intersection) increases (resp., decreases).  
 
Section~\ref{sec:volumetric} covers inequalities involving the intrinsic volumes of convex bodies in Euclidean $d$-space obtained as intersections of (possibly infinitely many) unit balls, which we call \emph{spindle convex sets}. 
In particular, we discuss Blaschke--Santal\'o type, and isoperimetric inequalities.
 
 Section~\ref{sec:spindle} covers spindle convex sets in the plane, with special attention paid to their approximation by the spindle convex hull of a finite subset. 

 A \emph{ball-polyhedron} is a  convex body obtained as the intersection of finitely many unit balls in Euclidean $d$-space, $\Ed$. In Section~\ref{sec:combinatorics}, we consider the combinatorial structure of their faces, and Steinitz-type questions on the realizability of graphs as the edge graph of (some type of) ball-polyhedron in $\Ethree$.
 
 Finally, in Section~\ref{sec:random}, we discuss the distribution of the intrinsic volume of a ball polyhedron obtained by choosing the centers of the balls randomly in $\Ed$, and an information theoretic approach to the Kneser--Poulsen Conjecture.

\subsection{Notation}
Let $\|\cdot \|$ denote the standard Euclidean norm of the $d$-dimensional Euclidean space $\mathbb{E}^{d}$. So, if ${\bf p}_i, {\bf p}_j$ are two points in $\mathbb{E}^{d}$, then $\|{\bf p}_i- {\bf p}_j \|$ denotes the Euclidean distance between them. It is convenient to denote a (finite) point configuration consisting of points ${\bf p}_1, {\bf p}_2, \dots, {\bf p}_N$ in $\mathbb{E}^{d}$ by ${\bf p}=({\bf p}_1, {\bf p}_2, \dots, {\bf p}_N)$. Now, if ${\bf p}=({\bf p}_1, {\bf p}_2, \dots, {\bf p}_N)$ and ${\bf q}=({\bf q}_1, {\bf q}_2, \dots, {\bf q}_N)$ are two configurations of $N$ points in $\mathbb{E}^{d}$ such that for all $1\le i<j\le N$ the inequality $\|{\bf q}_i- {\bf q}_j \|\le \|{\bf p}_i- {\bf p}_j \|$ holds, then we say that ${\bf q}$ is a {\it contraction} \index{contraction} of ${\bf p}$. 

Let $\ivol[d]{\cdot}$ denote the $d$-dimensional volume (Lebesgue measure) in $\Ed$. Let $\mathbf{B}^d[{\bf p}, r]$ denote the (closed) $d$-dimensional Euclidean ball centered at a point ${\bf p}$ with radius $r\geq 0$ in $\mathbb{E}^{d}$. In particular, let $\omega_d:=\ivol[d]{\mathbf{B}^d[{\bf p}, 1]}=\frac{\pi^{\frac{d}{2}}}{\Gamma(1+\frac{d}{2})}$. For simplicity, we set $\ivol[d]{\emptyset}=0$.

\section{From the Kneser--Poulsen conjecture to intersections of balls}\label{sec:KneserPoulsen}

\subsection{The Kneser--Poulsen conjecture}
\label{one}

In 1954, E. T. Poulsen \cite{Po} and in 1955, M. Kneser \cite{Kn} independently conjectured the following inequality for the case when $r_1=\dots=r_N$. Hence, following \cite{KlWa}, we attribute the following conjecture to Kneser and Poulsen.

\begin{Conjecture}
\label{elso}
If ${\bf q}=({\bf q}_1, {\bf q}_2, \dots, {\bf q}_N)$ is a contraction \index{contraction} of ${\bf p}=({\bf p}_1, {\bf p}_2,$ $ \dots, {\bf p}_N)$ in $\mathbb{E}^{d}$, then
\begin{equation}\label{K-P}
\ivol[d]{\bigcup_{i=1}^{N}\mathbf{B}^d[{\bf p}_i, r_i]}\ge \ivol[d]{\bigcup_{i=1}^{N}\mathbf{B}^d[{\bf q}_i, r_i]}
\end{equation}
holds for all $r_1>0, r_2> 0, \dots, r_N> 0$.
\end{Conjecture}

It is customary to attribute also the following closely related conjecture to Kneser and Poulsen. 

\begin{Conjecture}
\label{masodik}
If ${\bf q}=({\bf q}_1, {\bf q}_2,$ $ \dots, {\bf q}_N)$ is a contraction \index{contraction} of ${\bf p}=({\bf p}_1, {\bf p}_2,$ $ \dots, {\bf p}_N)$ in $\mathbb{E}^{d}$, then
\begin{equation}\label{G-K-W}
\ivol[d]{\bigcap_{i=1}^{N}\mathbf{B}^d[{\bf p}_i, r_i]}\le \ivol[d]{\bigcap_{i=1}^{N}\mathbf{B}^d[{\bf q}_i, r_i]}.
\end{equation}
holds for all $r_1>0, r_2> 0, \dots, r_N> 0$.
\end{Conjecture}

We note the following in connection with (\ref{G-K-W}). First, in 1979, Klee \cite{Kl} asked whether (\ref{G-K-W}) holds for $d=2$. Then in 1987, Gromov \cite{Gr87} published a proof of (\ref{G-K-W}) for all $N\leq d+1$ and $d>1$, and conjectured that his result extends to spherical $d$-space $\mathbb{S}^d$ (resp., hyperbolic $d$-space $\mathbb{H}^d$) for all $d>1$. Finally, in 1991, Klee and Wagon \cite{KlWa} asked whether (\ref{G-K-W}) holds for all $d>1$. In any case, following the tradition, we call Conjecture~\ref{elso} (resp., Conjecture~\ref{masodik}) the {\it Kneser--Poulsen conjecture} for unions of balls (resp., intersections of balls). 

\subsection{The Kneser--Poulsen conjecture for continuous contractions}
\label{two}

If ${\bf q}$ is a contraction of ${\bf p}$, then there may or may not be a continuous motion ${\bf p}(t)=({\bf p}_1(t), {\bf p}_2(t), \dots, {\bf p}_N(t))$, with  ${\bf p}_i(t)\in \mathbb{E}^{d}$ for all $0\le t\le 1$ and $1\le i\le N$ such that ${\bf p}(0)={\bf p}$ and ${\bf p}(1)={\bf q}$, and $\|{\bf p}_i(t)- {\bf p}_j(t)\|$ is monotone decreasing for all $1\le i<j\le N$. When there is such a motion, we say that ${\bf q}$ is a {\it continuous contraction} \index{continuous contraction} of ${\bf p}$. For example, a contraction that fixes the vertices of a $d$-dimensional simplex in $\mathbb{E}^{d}$ and maps an exterior point lying sufficiently far from the vertices onto an interior point of the simplex cannot be represented as a continuous contraction in $\mathbb{E}^{d}$. By representing the union (resp., intersection) of balls as a tiling of nearest (resp., farthest) point truncated Voronoi cells, Csik\'os \cite{Cs1} proved a volume formula for unions (resp., intersections) of balls, from which the Kneser--Poulsen conjecture follows for continuous contractions in a straightforward way.

\begin{Theorem}\label{Cs}
If ${\bf q}=({\bf q}_1, {\bf q}_2, \dots, {\bf q}_N)$ is a continuous contraction of ${\bf p}=({\bf p}_1, {\bf p}_2,$ $ \dots, {\bf p}_N)$ in $\mathbb{E}^{d}$, $d>1$, then (\ref{K-P}) and (\ref{G-K-W}) hold for all $r_1>0, r_2> 0, \dots, r_N> 0$.
\end{Theorem}

We note that the planar case of Theorem~\ref{Cs} has been proved independently in \cite{Bo68}, \cite{Cs0}, \cite{Ca}, and \cite{BeSa}. Moreover, Csik\'os \cite{Cs2} managed to generalize his volume formula to configurations of balls called {\it flowers} which are sets obtained from balls with the help of operations $\cap$ and $\cup$. This work extends to hyperbolic as well as spherical space. In particular, the union (resp., intersection) Kneser--Poulsen conjecture is true for continuous contractions in hyperbolic and spherical spaces. However, one cannot expect the Kneser--Poulsen conjectures to be true in spaces of non-constant curvature as shown in \cite{CsiKun}. On the other hand, Csik\'os \cite{Cs3} has succeeded in proving a Schl\"afli-type formula for polytopes with curved faces lying in pseudo-Riemannian Einstein manifolds, which can be used to provide another proof of Theorem~\ref{Cs} (for more details see \cite{Cs3}).

\subsection{The Kneser--Poulsen conjecture in the plane}
\label{three}

In \cite{BeCo}, the first named author and Connelly proved Conjecture~\ref{elso} as well as Conjecture~\ref{masodik} in the Euclidean plane. Thus, we have the following theorem.

\begin{Theorem}\label{26}
If ${\bf q}=({\bf q}_1, {\bf q}_2, \dots, {\bf q}_N)$ is a contraction \index{contraction} of ${\bf p}=({\bf p}_1, {\bf p}_2,$ $ \dots, {\bf p}_N)$ in $\mathbb{E}^{2}$, then
$$\ivol[2]{\bigcup_{i=1}^{N}\mathbf{B}^2[{\bf p}_i, r_i]}\ge \ivol[2]{\bigcup_{i=1}^{N}\mathbf{B}^2[{\bf q}_i, r_i]}$$
and
$$\ivol[2]{\bigcap_{i=1}^{N}\mathbf{B}^2[{\bf p}_i, r_i]}\le \ivol[2]{\bigcap_{i=1}^{N}\mathbf{B}^2[{\bf q}_i, r_i]}$$
hold for all $r_1>0, r_2> 0, \dots, r_N> 0$.
\end{Theorem}

In fact, the paper \cite{BeCo} contains a proof of an extension of the above theorem to flowers as well. In what follows we give an outline of the three-step proof published in \cite{BeCo} by phrasing it through a sequence of theorems each being higher-dimensional. Voronoi cells \index{Voronoi cell} play an essential role in the proofs of Theorems \ref{27} and \ref{28}.

\begin{Theorem}\label{27}

Consider $N$ moving closed $d$-dimensional balls $\mathbf{B}^d[{\bf p}_i(t), r_i]$ with $1\le i\le N, 0\le t\le 1$ in $\mathbb{E}^{d}, d\ge 2$. If $F_i(t)$ is the contribution of the $i$th ball to the boundary of the union $\bigcup_{i=1}^{N}\mathbf{B}^d[{\bf p}_i(t), r_i]$ (resp., of the intersection $\bigcap_{i=1}^{N}\mathbf{B}^d[{\bf p}_i(t), r_i]$), then
$$\sum_{1\le i\le N}\frac{1}{r_{i}}\ {\rm svol}_{d-1}\left(F_{i}(t)\right)$$
decreases (resp., increases) in t under any analytic contraction \index{analytic contraction} ${\bf p}(t)$ of the center points, where $0\le t\le 1$ and ${\rm svol}_{d-1}(\cdot)$ refers to the relevant $(d-1)$-dimensional surface volume. \index{surface volume}
\end{Theorem}

\begin{Theorem} \label{28}

Let the centers of the closed $d$-dimensional balls $\mathbf{B}^d[{\bf p}_i, r_i]$, $1\le i\le N$ lie in a $(d-2)$-dimensional affine subspace $L$ of $\mathbb{E}^{d}, d\ge 3$. If $F_i$ stands for the contribution of the $i$th ball to the boundary of the union $\bigcup_{i=1}^{N}\mathbf{B}^d[{\bf p}_i, r_i]$ (resp., of the intersection $\bigcap_{i=1}^{N}\mathbf{B}^d[{\bf p}_i, r_i]$), then
$$\ivol[d-2]{\bigcup_{i=1}^{N}\mathbf{B}^{d-2}[{\bf p}_i, r_i]}=
\frac{1}{2\pi}\sum_{1\le i\le N}\frac{1}{r_{i}}\:{\rm svol}_{d-1}(F_{i})
$$
$$\left( resp., \ \ivol[d-2]{\bigcap_{i=1}^{N}\mathbf{B}^{d-2}[{\bf p}_i, r_i]}=
\frac{1}{2\pi}\sum_{1\le i\le N}\frac{1}{r_{i}}\:{\rm svol}_{d-1}(F_{i})\right),
$$
where $\mathbf{B}^{d-2}[{\bf p}_i, r_i]=\mathbf{B}^d[{\bf p}_i, r_i]\cap L, 1\le i\le N$.

\end{Theorem}

\begin{Theorem} \label{29}
If ${\bf q}=({\bf q}_1, {\bf q}_2, \dots, {\bf q}_N)$ is a contraction \index{contraction} of ${\bf p}=({\bf p}_1, {\bf p}_2, \dots,$ $ {\bf p}_N)$ in $\mathbb{E}^{d}, d\ge 1$, then there is an analytic contraction \index{analytic contraction} $\mathbf{p}(t) = (\mathbf{p}_1(t), \dots,$ $\mathbf{p}_N(t)), 0\le t\le 1$ in $\mathbb{E}^{2d}$ such that $\mathbf{p}(0)=\mathbf{p}$ and $\mathbf{p}(1)=\mathbf{q}$.
\end{Theorem}

As for $d=2$ one has $d+2=2d$ therefore clearly, Theorems \ref{27}, \ref{28}, and \ref{29} imply Theorem~\ref{26}. Next, we mention two corollaries obtained from the above outlined proof and published in \cite{BeCo}. (For more details see \cite{BeCo}.)

\begin{Corollary}
Let ${\bf p}=({\bf p}_1, {\bf p}_2, \dots, {\bf p}_N)$ and ${\bf q}=({\bf q}_1, {\bf q}_2, \dots, {\bf q}_N)$ be two point configurations in $\mathbb{E}^{d}$ such that ${\bf q}$ is a contraction \index{analytic contraction} of ${\bf p}$ via the piecewise-analytic contraction $\mathbf{p}(t) = (\mathbf{p}_1(t), \dots,$ $\mathbf{p}_N(t)), 0\le t\le 1$ in $\mathbb{E}^{d+2}$ such that $\mathbf{p}(0)=\mathbf{p}$ and $\mathbf{p}(1)=\mathbf{q}$. Then the conclusions of Conjecture~\ref{elso} as well as Conjecture~\ref{masodik} hold in $\mathbb{E}^{d}$.
\end{Corollary}

The following generalizes a result of Gromov in \cite{Gr87}, who proved it in the case $N\le d+1$.

\begin{Corollary}\label{Be-Co-Gr} 
If ${\bf q}=({\bf q}_1, {\bf q}_2, \dots, {\bf q}_N)$ is an arbitrary contraction \index{contraction} of
${\bf p}=({\bf p}_1, {\bf p}_2, \dots, {\bf p}_N)$ in $\mathbb{E}^{d}$ and $N\le d+3$, then both Conjecture~\ref{elso} and Conjecture~\ref{masodik} hold.
\end{Corollary}

Applying results on central sets of unions of finitely many balls, Gorbovickis (\cite{Go2018}) has proved the following analogues of the Kneser--Poulsen conjecture in $\Sph^2$ (resp., $\HH^2$).

\begin{Theorem}\label{gor-main}
If the union of a finite set of closed disks in $\Sph^2$ (resp., $\HH^2$) has a simply connected interior, then the area of the union of these disks cannot increase after any contractive rearrangement.
\end{Theorem}

The following statement published by Gorbovickis in \cite{Go2018}, is a straightforward corollary of Theorem~\ref{gor-main}. 

\begin{Corollary}\label{gor-cor}
\item (i) If the intersection of a finite set of closed disks in $\Sph^2$ is connected, then the area of the intersection of these disks cannot decrease after any contractive rearrangement.
\item (ii) The area of the intersection of finitely many closed disks having radii at most $\frac{\pi}{2}$ in $\Sph^2$ cannot decrease after any contractive rearrangement.
\end{Corollary}

It is surprising that part (ii) of Corollary~\ref{gor-cor} can be extended to $\Sph^d$, $d>2$ for a specific radius of balls (i.e., spherical caps). Namely, the first named author and Connelly 
\cite{BeCo04} proved the following analogue of the Kneser--Poulsen conjecture in $\Sph^d$.

\begin{Theorem}\label{Be-Co}
If a finite set of closed spherical caps of (angular) radius $\frac{\pi}{2}$ (i.e., of
closed hemispheres) in $\Sph^d$, $d>1$ is rearranged
so that the (spherical) distance between each pair of centers does not increase, then
the (spherical) d-dimensional volume of the intersection does not decrease and the
(spherical) d-dimensional volume of the union does not increase.
\end{Theorem}

Finally, we mention the following recent result of Csik\'os and Horv\'ath (\cite{Cs2018}) that extends part (ii) of Corollary~\ref{gor-cor} to $\HH^2$ by applying the method of Gorbovickis (\cite{Go2018}) to co-central sets, which one can regard as natural duals of central sets. For properties of central (resp., co-central) sets of compact sets, which might be of independent interest, see the relevant sections of \cite{Go2018} (resp., Csik\'os and Horv\'ath \cite{Cs2018}). 
\begin{Theorem}\label{csik-main}
The area of the intersection of finitely many closed disks in $\HH^2$ cannot decrease after any contractive rearrangement.
\end{Theorem} 

\subsection{Kneser--Poulsen-type results in Minkowski spaces}
\label{four}

Let $\KK\subset \Re^d$ be an $\oo$-symmetric convex body, i.e., a compact convex set with nonempty interior symmetric about the origin $\oo$ in $\Re^d$. Let $\| \cdot \|_{\KK}$ denote the norm generated by $\KK$, which is defined by $\|\xx\|_{\KK}:=\min\{\lambda \geq 0\  |\  \xx\in \lambda \KK\}$ for $\xx\in \Re^d$. Furthermore, let us denote $\Re^d$ with the norm $\| \cdot \|_{\KK}$ by $\M^d_{\KK}$ and call it the {\it Minkowski space of dimension $d$ generated by} $\KK$. We write $\mathbf{B}_{\KK}[\xx, r]:=\xx+r\KK$ for $\xx\in\Re^d$ and $r>0$ and call any such set a (closed) {\it ball of radius $r$} in $\M^d_{\KK}$, where $+$ refers to vector addition extended to subsets of $\Re^d$ in the usual way (called the Minkowski sum). Next, recall the following definitions from \cite{Be20}. 

\begin{Definition}\label{r-ball bodies and r-ball polyhedra}
For $X\subseteq\Re^d$ and $r>0$ let
$$X_{r}^{\KK}:=\bigcup\{\mathbf{B}_{\KK}[\xx,r]\ |\ \xx\in X\}\ ({\rm resp.,}\ X_{\KK}^{r}:=\bigcap\{\mathbf{B}_{\KK}[\xx,r]\ |\ \xx\in X\})$$
denote the {\rm $r$-ball neighbourhood} of $X$ (resp., {\rm $r$-ball body} generated by $X$) in $\M^d_{\KK}$. If $X\subset\Re^d$ is a finite set, then we call
$X_{r}^{\KK}$ (resp., $X_{\KK}^{r}$) the {\rm $r$-ball molecule} (resp., {\rm $r$-ball polyhedron}) generated by $X$ in $\M^d_{\KK}$.
We simplify our notations when $\KK$ is a Euclidean ball of $\Re^d$ as follows: for a set $ X\subseteq\Ee^d$, $d>1$ and $r>0$, let 
$$X_r:=\bigcup_{\xx\in X}\mathbf{B}^d[\xx, r] \ \text{(resp.,}\  X^r:=\bigcap_{\xx\in X}\mathbf{B}^d[\xx, r]).$$
\end{Definition}

Following \cite{Be20}, we say that a (labeled) point set $Q:=\{\qqq_1, \dots , \qqq_N\}\subset\Re^d$ is a {\it uniform contraction} of a (labeled) point set $P:=\{\pp_1, \dots ,\pp_N\}\subset\Re^d$ with {\it separating value} $\lambda>0$ in $\M^d_{\KK}$  if
$$\|\qqq_i-\qqq_j\|_{\KK}\leq\lambda\leq\|\pp_i-\pp_j\|_{\KK}\ {\rm holds}\ {\rm for}\  {\rm all}\ 1\leq i<j\leq N.$$ Finally,  the (labeled) point set $Q:=\{\qqq_1, \dots , \qqq_N\}\subset\Re^d$ is a {\it uniform contraction} of the (labeled) point set $P:=\{\pp_1, \dots ,\pp_N\}\subset\Re^d$ in $\M^d_{\KK}$ if there exists a $\lambda>0$ such that $Q$ is a uniform contraction of $P$ with separating value $\lambda$ in $\M^d_{\KK}$. The following theorem of \cite{Be20} is an analogue of the Kneser--Poulsen conjecture for unions of balls under uniform contractions in $\M^d_{\KK}$. (For an earlier and weaker Euclidean version of this theorem we refer the interested reader to \cite{BeNa18}.)
\begin{Theorem}\label{bez-union} Let $\KK$ be an $\oo$-symmetric convex body in $ \Re^d$.
If $r>0, d>1$, $N\geq 2^d$, and $Q:=\{\qqq_1, \dots , \qqq_N\}\subset\Re^d$ is a uniform contraction of $P:=\{\pp_1, \dots ,\pp_N\}\subset\Re^d$ in $\M^d_{\KK}$, then
\begin{equation}\label{bez-main-1}
\ivol[d]{Q_r^{\KK}}\leq \ivol[d]{P_r^{\KK}}.
\end{equation}
\end{Theorem}

As the proof of Theorem~\ref{bez-union} in \cite{Be20} is a short one and might be of independent interest, we recall it here. The details are as follows. Let $\lambda>0$ such that $Q$ is a uniform contraction of $P$ with separating value $\lambda$ in $\M^d_{\KK}$.
As (\ref{bez-main-1}) holds trivially for $0<r\leq \frac{\lambda}{2}$ therefore we may assume that $0<\frac{\lambda}{2}<r$. Recall that for a bounded set  $\emptyset\neq X\subset\Re^d$
the diameter ${\rm diam}_{\KK}(X)$ of $X$ in $\M^d_{\KK}$ is defined by ${\rm diam}_{\KK}(X):=\sup\{\|\xx_1-\xx_2\|_{\KK} \ | \ \xx_1,\xx_2\in X\}$. Clearly,
\begin{equation}\label{union-1}
{\rm diam}_{\KK}(Q_r^{\KK})={\rm diam}_{\KK}(Q)+2r\leq \lambda+2r .
\end{equation}
Thus, the isodiametric inequality in Minkowski spaces (Theorem 11.2.1 in \cite{BuZa}) and (\ref{union-1}) imply that
\begin{equation}\label{union-2}
\ivol[d]{Q_r^{\KK}}\leq \left(r+\frac{\lambda}{2}\right)^d\ivol[d]{\KK} .
\end{equation}

\noindent For the next estimate recall that the {\it volumetric radius relative to $\KK$} of a nonempty compact set $A\subset\Re^d$ is denoted by $r_{\KK}(A)$ and it is defined by $\ivol[d]{r_{\KK}(A)\KK}=\left(r_{\KK}(A)\right)^d\ivol[d]{\KK}:=\ivol[d]{A}$. Using this concept one can derive the following inequality from the Brunn--Minkowski inequality in a rather straightforward way (Theorem 9.1.1 in \cite{BuZa}):
\begin{equation}\label{union-3}
r_{\KK}(A_{\epsilon}^{\KK})\geq r_{\KK}(A)+\epsilon ,
\end{equation}
which holds for any $\epsilon>0$. As $\{\mathbf{B}_{\KK}\left[\pp_i, \frac{\lambda}{2}\right]\ |\ 1\leq i\leq N\}$ is a packing in $\Re^d$ therefore $r_{\KK}\left(P_{\frac{\lambda}{2}}^{\KK}\right)=N^{\frac{1}{d}}\frac{\lambda}{2}$. Combining this observation with (\ref{union-3}) yields

\begin{equation}\label{union-4}
\ivol[d]{P_r^{\KK}}=\ivol[d]{ \left(P_{\frac{\lambda}{2}}^{\KK}\right)_{r-\frac{\lambda}{2}}^{\KK}}\geq\left(N^{\frac{1}{d}}\frac{\lambda}{2}+\left(r-\frac{\lambda}{2}\right) \right)^d\ivol[d]{\KK}=
\end{equation}
$$\left(r+\left(N^{\frac{1}{d}}-1\right)\frac{\lambda}{2} \right)^d\ivol[d]{\KK} .$$

\noindent Finally, as $N\geq 2^d$ therefore $N^{\frac{1}{d}}-1\geq 1$ and Theorem~\ref{bez-union} follows from (\ref{union-2}) and (\ref{union-4}) in a straightforward way.

Thus, Corollary~\ref{Be-Co-Gr} and Theorem~\ref{bez-union} prove Conjecture~\ref{elso} for uniform contractions of sufficiently many congruent balls.

\begin{Corollary}\label{K-P for uniform contractions}
Let $r>0, d>2$, and let either $1\leq N\leq d+3$ or $N\geq 2^d$. If $Q:=\{\qqq_1, \dots , \qqq_N\}\subset\Re^d$ is a uniform contraction of $P:=\{\pp_1, \dots ,\pp_N\}\subset\Re^d$ in $\Ee^d$, then
\begin{equation}
\ivol[d]{Q_r}\leq \ivol[d]{P_r}.
\end{equation}
\end{Corollary}

Recall from \cite{Sc14} that the compact convex set $\emptyset\neq{A'}\subset\Re^d$ is a {\it summand} of the compact convex set  $\emptyset\neq{A}\subset\Re^d$ if there exists a compact convex set  $\emptyset\neq{A''}\subset\Re^d$ such that ${A'}+{A''}={A}$. Furthermore, the convex body $\mathbf{B}\subset\Re^d$ is a {\it generating set} if any nonempty intersection of translates of $\mathbf{B}$ is a summand of $\mathbf{B}$. In particular, we say that {\it $\M^d_{\KK}$ possesses a generating unit ball} if $\mathbf{B}_{\KK}[\oo, 1]=\KK$ is a generating set in $\Re^d$. Here we recall the following statements only. Two-dimensional convex bodies are generating sets. Euclidean balls are generating sets as well and the system of generating sets is stable under non-degenerate linear maps and under direct sums. We refer the interested reader for additional properties of generating  sets to the recent paper \cite{Bu25} and the references mentioned there. The following theorem of \cite{Be20} is an analogue of the Kneser--Poulsen conjecture for intersections of balls under uniform contractions in $\M^d_{\KK}$. 
\begin{Theorem}\label{bez-intersection}
Let $r>0, d>1$, $N\geq 3^d$, and let the $\oo$-symmetric convex body $\KK$ be a generating set in $\Re^d$. If $Q:=\{\qqq_1, \dots , \qqq_N\}\subset\Re^d$ is a uniform contraction of $P:=\{\pp_1, \dots ,\pp_N\}\subset\Re^d$ in $\M^d_{\KK}$, then
\begin{equation}\label{bez-main-2}
\ivol[d]{P_{\KK}^{r}}\leq \ivol[d]{Q_{\KK}^{r}}.
\end{equation}
\end{Theorem}

In what follows we prove Theorem~\ref{bez-intersection} based on \cite{Be20}. As a Blaschke-Santal\'o-type inequality we are going to use Theorem~\ref{BS-type inequality extended}, which might be of independent interest. In order to state it we need to recall the following concepts. Let $\emptyset\neq A\subset\Ee^d$ be a compact convex set, and $1\leq k\leq d$. We denote the {\it $k$-th intrinsic volume} of $A$ by $\ivol{A}$. It is well known that $\ivol[d]{A}$ is the $d$-dimensional 
volume of $A$, $2\ivol[d-1]{A}$ is the surface area of $A$, and $\frac{2\omega_{d-1}}{d\omega_d}\ivol[1]{A}$ is equal to the mean width of $A$. We set $\ivol[k]{\emptyset}:=0$ for $1\leq k\leq d$. Next, let $\KK\subset \Re^d$ be an $\oo$-symmetric convex body in $\Re^d$. For a set $\emptyset\neq X\subset \mathbf{B}_{\KK}[\xx, r]$ with $\xx\in\Re^d$ and $r>0$ let ${\rm cr}_{\KK}(X):=\inf \{R>0 \ |\ X\subseteq B_{\KK}[\xx', R]\ {\rm with}\ \xx'\in\Re^d\}\leq r$. We call ${\rm cr}_{\KK}(X)$ the {\it circumradius} of $X$ in $\M^d_{\KK}$ and we call ${\rm conv}_{r, \KK}(X):=\bigcap\{B_{\KK}[\xx', r]\ |\ X\subseteq B_{\KK}[\xx' , r]\ {\rm with}\ \xx'\in\Re^d\}$
the {\it $r$-ball convex hull} of $X$ in $\M^d_{\KK}$. Finally, recall the following easy extension of Lemma 5 from $\Ee^d$ to $\M^d_{\KK}$ in \cite{Be18}:
\begin{equation}\label{important-trivial}
X_{\KK}^r=\left( {\rm conv}_{r, \KK}(X)\right)^r .
\end{equation}

\begin{Theorem}\label{BS-type inequality extended}
Let $r>0, d>1$, $1\leq k\leq d$, and let the $\oo$-symmetric convex body $\KK$ be a generating set in $\Re^d$. If $\emptyset\neq A\subset\Re^d$ is a compact set with  ${\rm cr}_{\KK}(A)\leq r$ and $B\subset\Re^d$ is a ball of $\M^d_{\KK}$ such that $\ivol[k]{{\rm conv}_{r, \KK}A}=\ivol[k]{B}$, then $\ivol[k]{A_{\KK}^r}\leq \ivol[k]{B_{\KK}^r} $.
\end{Theorem}
\begin{proof}
We start by recalling Lemma 8 from \cite{Be20}. 

\begin{Lemma}\label{special}
Let $d>1$, $r>0$, and let $\M^d_{\KK}$ possess a generating unit ball. If $X_{\KK}^r\neq\emptyset$, then
\begin{equation}\label{intersection-5}
X_{\KK}^r-{\rm conv}_{r, \KK}(X)=B_{\KK}[\oo, r] .
\end{equation}
\end{Lemma}

Clearly, the Brunn-Minkowski inequality for intrinsic volumes (\cite{Gar}, Eq. (74)) combined with Lemma~\ref{special} implies

\begin{Corollary}\label{special-corollary}
Let $r>0, d>1$, $1\leq k\leq d$, and let $\M^d_{\KK}$ possess a generating unit ball. If $X_{\KK}^r\neq\emptyset$, then
\begin{equation}\label{intersection-6}
\ivol[k]{X_{\KK}^r}^{\frac{1}{k}}+\ivol[k]{{\rm conv}_{r,\KK}(X)}^{\frac{1}{k}}\leq r \ivol[k]{\KK}^{\frac{1}{k}} . 
\end{equation}
\end{Corollary}

Now, we are ready to prove Theorem~\ref{BS-type inequality extended}. By assumption $A_{\KK}^r\neq\emptyset$. For the estimates that follow, let the {\it k-th intrinsic volumetric radius (relative to $\KK$}) of $A$ be denoted by $r_{k, \KK}(A)$ and defined by $\ivol[k]{r_{k, \KK}(A)\KK}=\left(r_{k, \KK}(A)\right)^k\ivol[k]{\KK}:=\ivol[k]{{\rm conv}_{r, \KK}A}$. Notice that $r_{k, \KK}(A)\KK=B$. Thus, Corollary~\ref{special-corollary} and homogeneity
(of degree k) of k-th intrinsic volume imply in a straightforward way that
$$
\ivol[k]{A_{\KK}^r}\leq\left[ r\ivol[k]{\KK}^{\frac{1}{k}}-\ivol[k]{{\rm conv}_{r,\KK}(A)}^{\frac{1}{k}} \right]^k=\left[ r \ivol[k]{\KK}^{\frac{1}{k}}-\ivol[k]{r_{k, \KK}(A)\KK}^{\frac{1}{k}} \right]^k$$
$$
=(r-r_{k, \KK}(A))^k\ivol[k]{\KK}=\ivol[k]{(r-r_{k, \KK}(A))\KK}=\ivol[k]{B_{\KK}^r},
$$
finishing the proof of Theorem~\ref{BS-type inequality extended}.
\end{proof}

Based on Theorem~\ref{BS-type inequality extended} we can give a short proof for Theorem~\ref{bez-intersection} as follows. If $r\leq {\rm cr}_{\KK}(P)$, then clearly $V_d(P_{\KK}^r)=0\leq V_d(Q_{\KK}^r)$, finishing the proof of Theorem~\ref{bez-intersection} in this case. Hence, for the rest we may assume that ${\rm cr}_{\KK}(P)<r$. By assumption there exists $\lambda>0$ such that $Q$ is a uniform contraction of $P$ with separating value $\lambda$ in $\M^d_{\KK}$. Then $\emptyset\neq P_{\KK}^r=\left(P_{\frac{\lambda}{2}}^{\KK}\right)_{\KK}^{r+\frac{\lambda}{2}}$ with $\ivol[d]{P_{\frac{\lambda}{2}}^{\KK}}=N\left(\frac{\lambda}{2}\right)^d\ivol[d]{\KK}=\ivol[d]{N^{\frac{1}{d}}\frac{\lambda}{2}\KK}$ and so, Theorem~\ref{BS-type inequality extended} for $k=d$  and (\ref{important-trivial}) imply that
$$
\ivol[d]{P_{\KK}^r}=\ivol[d]{\left(P_{\frac{\lambda}{2}}^{\KK}\right)_{\KK}^{r+\frac{\lambda}{2}}}\leq \ivol[d]{\left(N^{\frac{1}{d}}\frac{\lambda}{2}\KK\right)_{\KK}^{r+\frac{\lambda}{2}}}=\ivol[d]{\left(r+\frac{\lambda}{2}-N^{\frac{1}{d}}\frac{\lambda}{2}\right)\KK}
$$
\begin{equation}\label{final upper bound}
=\left(r-(N^{\frac{1}{d}}-1)\frac{\lambda}{2} \right)^d\ivol[d]{\KK}.
\end{equation}

Next, recall that $Q=\{\qqq_1,\dots ,\qqq_N\}\subset\Re^d$ with $N\geq 3^d$ such that $\|\qqq_i-\qqq_j\|_{\KK}\leq \lambda$ holds for all $1\leq i<j\leq N$. Thus, Bohnenblust's theorem (Theorem 11.1.3 in \cite{BuZa}) yields ${\rm cr}_{\KK}(Q)\leq \frac{d}{d+1}{\rm diam}_{\KK}(Q)\leq\frac{d}{d+1}\lambda$, from which it is easy to derive that
\begin{equation}\label{intersection-3}
\ivol[d]{Q_{\KK}^r}\geq \left(r-\frac{d}{d+1}\lambda\right)^d\ivol[d]{\KK} .
\end{equation}

Finally, observe that $N\geq 3^d$ implies
$\left(r-(N^{\frac{1}{d}}-1)\frac{\lambda}{2} \right)^d\ivol[d]{\KK}\leq \left(r-\lambda\right)^d\ivol[d]{\KK}<\left(r-\frac{d}{d+1}\lambda\right)^d\ivol[d]{\KK}$. This inequality combined with (\ref{final upper bound}) and (\ref{intersection-3}) completes the proof of Theorem~\ref{bez-intersection}.

The following theorem of \cite{BeNa18} strengthens Theorem~\ref{bez-intersection} using isoperimetric properties of balls in $\Ee^d$. (See also \cite{Be22} for a Blaschke--Santal\'o-type inequality on intrinsic volumes of $r$-ball bodies in $\Ee^d$, which can be used to prove Theorem~\ref{K-P for uniform contractions by B-N} in a somewhat different way.) 

\begin{Theorem}\label{K-P for uniform contractions by B-N}
Let $r>0$, $d>1$, $1\leq k\leq d$, and let $N\geq (1+\sqrt{2})^d$. If $Q:=\{\qqq_1, \dots , \qqq_N\}\subset\Re^d$ is a uniform contraction of $P:=\{\pp_1, \dots ,\pp_N\}\subset\Re^d$ in $\Ee^d$, then
\begin{equation}\label{bez-main-11}
\ivol[k]{P^r}\leq \ivol[k]{Q^r}.
\end{equation}
\end{Theorem}

This leads to the following volumetric question on $r$-ball bodies, which was raised in \cite{BeNa18} under the belief that the Kneser--Poulsen conjecture for intersections of congruent balls can be sharpened.

\begin{Problem}\label{Al-Be-Na}
Let $r>0$, $d>1$, $1\leq k\leq d-1$, and let $N>1$. If $Q:=\{\qqq_1, \dots , \qqq_N\}\subset\Re^d$ is a contraction of $P:=\{\pp_1, \dots ,\pp_N\}\subset\Re^d$ in $\Ee^d$, then prove or disprove that
\begin{equation}\label{Al-Be-Na-question}
\ivol[k]{P^r}\leq \ivol[k]{Q^r}.
\end{equation}
\end{Problem}

Surprisingly, Problem~\ref{Al-Be-Na} is still open in $\Ee^2$. In fact, Alexander \cite{Al85} conjectured that under any contraction of the center points of a finite family of unit disks in the plane, the perimeter of the intersection does not decrease. For partial results on {\it Alexander's conjecture} we refer the interested reader to \cite{BeCoCs06}. We note that the higher dimensional ($d>2$) analogue of Alexander's conjecture (that is, Problem~\ref{Al-Be-Na} for $k=d-1$ and $d>2$) is wide open.

\section{More on volumetric properties of intersections of congruent balls}\label{sec:volumetric}

\subsection{Blaschke--Santal\'o-type inequalities for \texorpdfstring{$r$}{r}-ball bodies} Recall that according to Definition~\ref{r-ball bodies and r-ball polyhedra}, for a set $X\subseteq\Ee^d$, $d>1$ and $r>0$, the $r$-ball body generated by $X$ is $X^r=\bigcap_{\xx\in X}\mathbf{B}^d[\xx, r]$. We note that either $X^r=\emptyset$, or $X^r$ is a point, or ${\rm int} (X^r)\neq\emptyset$. Perhaps not surprisingly, $r$-ball bodies of $\Ee^d$ have already been investigated in a number of papers however, under various names such as ``\"uberkonvexe Menge'' (\cite{Ma}), ``$r$-convex domain'' (\cite{Fe}), ``spindle convex set'' (\cite{BLNP}, \cite{KMP}), ``ball convex set'' (\cite{LNT}), ``hyperconvex set'' ({\cite{FKV}), and ``$r$-dual set'' (\cite{Be18}). Next, we recall that the {\it $r$-ball convex hull} ${\rm conv}_rX$ of $\emptyset\neq X\subset\Ee^d$ for $d>1$ and $r>0$, is defined by $${\rm conv}_rX:=\bigcap\{ \mathbf{B}^d[\xx, r]\ |\ X\subseteq \mathbf{B}^d[\xx, r]\}.$$
Moreover, let ${\rm conv}_r \Ee^d:=\Ee^d$ and ${\rm conv}_r\emptyset:=\emptyset$. Thus, we say that  $X\subseteq\Ee^d$ is {\it $r$-ball convex} if $X={\rm conv}_rX$. Furthermore, we recall from \cite{BLNP} the following concepts. If $\|\xx-\yy\|\leq 2r$ for some $\xx, \yy\in \Ee^d$, then let the {\it (closed) $r$-spindle} $[\xx,\yy]_r$ be defined by $[\xx,\yy]_r:={\rm conv}_r\{\xx,\yy\}$. If $\|\xx-\yy\|> 2r$ for some $\xx, \yy\in \Ee^d$, then let $[\xx,\yy]_r:=\Ee^d$. So, we say that $X\subseteq\Ee^d$ is {\it $r$-spindle convex} \label{def:spindleconvex} if for any $\xx_1,\xx_2\in X$ we have that $[\xx_1,\xx_2]_r\subseteq X$. Finally, recall that if $\mathbf{B}^d[\xx, r]$ contains the set $\emptyset\neq X\subset\Ee^d$ with $\yy\in {\rm bd}(X)\cap {\rm bd}(\mathbf{B}^d[\xx, r])$, then we say that {\it $\mathbf{B}^d[\xx, r]$ supports $X$ at the point $\yy$}. 
When $r$ is omitted, it is taken as $r=1$. The following statement from \cite{BLNP} shows the equivalence of the concepts introduced.

\begin{Proposition}\label{equivalence}
Let $r>0$, $d>1$, and $\emptyset\ne \mathbf{A}\subset\Ee^d$ be a compact (linearly) convex set. Then the following are equivalent.
\item (1) $\mathbf{A}$ is an $r$-ball body.
\item (2) $\mathbf{A}$ is $r$-ball convex.
\item (3) $\mathbf{A}$ is $r$-spindle convex.
\item (4) For every boundary point of $\mathbf{A}$ there exits a closed ball of radius $r$ that supports $\mathbf{A}$ at that point.
\end{Proposition}

The way $r$-ball bodies have been introduced in Definition~\ref{r-ball bodies and r-ball polyhedra} suggests to investigate the map $\mathbf{A}\mapsto \mathbf{A}^r$ that assigns to an arbitrary $r$-ball body $\mathbf{A}$ another $r$-ball body namely, $\mathbf{A}^r$. This map has been investigated in \cite{BeCoCs06}, \cite{BLNP}, \cite{BM11}, \cite{FKV}, \cite{Be18}, and based on the properties derived there we call this map {\it $r$-duality}. The following statement highlights some remarkable properties of $r$-duality, which with the exception of (6) and (7) are analogues of the corresponding properties of the classical polarity map. If we write $X\vee X_1:= {\rm conv}_r\left( X\cup X_1\right)$ for sets $X, X_1 \subseteq\Ee^d$ and $\mathcal{A}^r_d$ for the {\it set of $r$-ball bodies} in $\Ee^d$, then we observe that $(\mathcal{A}^r_d, \cap, \vee)$ is a lattice.

\begin{Proposition}\label{basic properties}
If $X, X_1 \subseteq\Ee^d$ are arbitrary sets and $\mathbf{A}, \mathbf{A}_1\in \mathcal{A}^r_d$ are arbitrary $r$-ball bodies for $d>1$ and $r>0$, then we have the following:
\item (1) $\left(\mathbf{A}^r\right)^r=\mathbf{A}$.
\item (2) $\emptyset\neq X\subseteq X_1$ implies $X_1^r\subseteq X^r$.
\item (3) $X^r= ({\rm conv}_rX)^r$.
\item (4) $(X \cup X_1)^r=X^r\cap X_1^r$ and $\left(\mathbf{A} \vee \mathbf{A}_1\right)^r=\mathbf{A}^r\cap \mathbf{A}_1^r$.
\item (5) $\left(\mathbf{A}\cap\mathbf{A}_1\right)^r=\mathbf{A}^r\vee\mathbf{A}_1^r$.
\item (6)  $\mathbf{A}+(-\mathbf{A}^r)=\mathbf{B}^d[\oo, r]$ and $\mathbf{A}+\mathbf{A}^r$ is a body of constant width $2r$ provided that $\mathbf{A}\neq\emptyset$. In particular, the sum of the minimal width of $\mathbf{A}$ and the diameter of $\mathbf{A}^r$ is equal to $2r$.
\item (7) The largest ball contained in $\mathbf{A}$ and the smallest ball that contains $\mathbf{A}^r$ are concentric and the sum of their radii is equal to $r$.
\end{Proposition}

According to (1), (2), (4), and (5) of Proposition~\ref{basic properties} the $r$-duality $\mathbf{A}\mapsto \mathbf{A}^r$ is an order reversing involution that interchanges the lattice operations.

Next, recall the theorem of Gao, Hug, and Schneider \cite{GHSch} stating that for any convex body of given volume in $\mathbb{S}^d$ the volume of the spherical polar body is maximal if the convex body is a ball. In 2018, this theorem was extended by the first named author \cite{Be18} to $r$-ball bodies of spaces of constant curvature. While the spherical and hyperbolic versions can be proved using two-point symmetrization, and that method works for the Euclidean case as well, there is another and simpler way to prove the following statement. If $A\subset\Ee^d$, $d>1$ is a compact set of volume $\ivol[d]A>0$ and $r>0$ and $\mathbf{B}\subseteq\Ee^d$ is a ball with $\ivol[d]{A}=\ivol[d]{\mathbf{B}}$, then $\ivol[d]{A^r}\leq \ivol[d]{\mathbf{B}^r}$. In fact, \cite{Be22} shows that similar inequalities hold for intrinsic volumes (see also \cite{Be23}). Thus, we have the following Blaschke--Santal\'o-type inequality for $r$-ball bodies in $\Ee^d$.

\begin{Theorem}\label{B-S inequalities}
Let $d>1$, $r>0$, and $1\leq k\leq d$. If $\mathbf{A}\in \mathcal{A}^r_d$, $d>1$ is an $r$-ball body of volume $\ivol[d]{\mathbf{A}}>0$ and $\mathbf{B}\subset\Ee^d$ is a ball with $\ivol[d]{\mathbf{A}}=\ivol[d]{\mathbf{B}}$, then 
\begin{equation}\label{Bezdek-inequality-generalized}
\ivol[k]{\mathbf{A}^r}\leq \ivol[k]{\mathbf{B}^r}
\end{equation}
holds with equality if and only if $\mathbf{A}$ is congruent (i.e., isometric) to $\mathbf{B}$.
\end{Theorem} 

For the convenience of the reader we recall here the short proof of Theorem~\ref{B-S inequalities} from \cite{Be23}. Let $\mathbf{B}=\mathbf{B}^d[\oo, r-R]$ and $\mathbf{B}^r=\mathbf{B}^d[\oo, R]$ for $0<R< r$. The Brunn--Minkowski inequality for intrinsic volumes (\cite{Gar}, Eq. (74)) and (6) of Proposition~\ref{basic properties} imply
\begin{equation}\label{B-M-I}
 \ivol[k]{\mathbf{A}}^{\frac{1}{k}}+\ivol[k]{\mathbf{A}^r}^{\frac{1}{k}}= \ivol[k]{\mathbf{A}}^{\frac{1}{k}}+\ivol[k]{-\mathbf{A}^r}^{\frac{1}{k}}\leq  \ivol[k]{\mathbf{A}+(-\mathbf{A}^r)}^{\frac{1}{k}}=  \ivol[k]{\mathbf{B}^d[\oo, r]}^{\frac{1}{k}}
\end{equation}
with equality if and only if ($\mathbf{A}$ and $-\mathbf{A}^r$ are homothetic, i.e.,) $\mathbf{A}$ is congruent to $\mathbf{B}$. Finally, (\ref{B-M-I}), the isoperimetric inequality for intrinsic volumes stating that among convex bodies of given volume the balls have the smallest $k$-th intrinsic volume (\cite{Sc14}, Section 7.4), and the homogeneity (of degree $k$) of $k$-th intrinsic volume imply
$$
 \ivol[k]{\mathbf{A}^r}^{\frac{1}{k}}\leq  \ivol[k]{\mathbf{B}^d[\oo, r]}^{\frac{1}{k}}- \ivol[k]{\mathbf{A}}^{\frac{1}{k}}\leq  \ivol[k]{\mathbf{B}^d[\oo, r]}^{\frac{1}{k}}- \ivol[k]{\mathbf{B}^d[\oo, r-R]}^{\frac{1}{k}}
$$
$$
= \ivol[k]{\mathbf{B}^d[\oo, R]}^{\frac{1}{k}}= \ivol[k]{\mathbf{B}^r}^{\frac{1}{k}},
$$
with $ \ivol[k]{\mathbf{A}^r}^{\frac{1}{k}}= \ivol[k]{\mathbf{B}^r}^{\frac{1}{k}}$ if and only if $\mathbf{A}$ is congruent to $\mathbf{B}$. This completes the proof of Theorem~\ref{B-S inequalities}.

In \cite{Be23}, the following corollary has been derived from Theorem~\ref{B-S inequalities}, which for $k=d$ has been proved independently by Fodor, Kurusa, and V\'igh \cite{FKV}. 
\begin{Corollary}\label{Fodor-Kurusa-Vigh-extended}
Let $d>1$, $r>0$, $1\leq k\leq d$, and $\mathbf{A}\in \mathcal{A}^r_d$ be an $r$-ball body. Set $P_k(\mathbf{A}):=\ivol[k]{\mathbf{A}}\ivol[k]{\mathbf{A}^r}$. Then
\begin{equation}\label{F-K-V}
P_k(\mathbf{A})\leq P_k\left(\mathbf{B}^d\left[\oo, \frac{r}{2}\right]\right)
\end{equation}
with equality if and only if $\mathbf{A}$ is congruent to $\mathbf{B}^d[\oo, \frac{r}{2}]$.
\end{Corollary}
Indeed, Corollary~\ref{Fodor-Kurusa-Vigh-extended} follows from Theorem~\ref{B-S inequalities} by the homogeneity (of degree $k$) of $k$-th intrinsic volume, and the observation that $f(x)=x^k(r-x)^k$, $0\leq x\leq r$ has a unique maximum value at $x=\frac{r}{2}$ for any $d>1$, $1\leq k\leq d$, and $r>0$.

For the sake of completeness we note that Theorem~\ref{B-S inequalities} has been used in \cite{Be22} (see also \cite{Be18}) to prove the Kneser--Poulsen conjecture for uniform contractions of intersections of sufficiently many congruent balls. We close this section with the following complementary question to Theorem~\ref{B-S inequalities}, which has been raised in \cite{Be23}, and can be regarded as a Mahler-type problem for $r$-ball bodies.

\begin{Problem}\label{Mahler-type}
Let $d>2$, $1\leq k\leq d$, and $0<v<r^d\omega_d=\ivol[d]{\mathbf{B}^d[\oo, r]}$. Find the minimum of $\ivol[k]{\mathbf{A}^r}$ for all $r$-ball bodies $\mathbf{A}\in \mathcal{A}^r_d$ of given volume $v=\ivol[d]{\mathbf{A}}$.
\end{Problem}

\begin{Remark}\label{2D-Mahler-type}
Problem~\ref{Mahler-type} for $d=2$ was solved in \cite{Be23} as follows. Let $0<v<\pi r^2$. Then the minimum of $\ivol[1]{\mathbf{A}^r}$ (resp., $\ivol[2]{\mathbf{A}^r}$) for all $r$-disk domains $\mathbf{A}\in \mathcal{A}^r_2$ of given area $v=\ivol[2]{\mathbf{A}}$ is attained only for $r$-lenses, which are intersections of two disks of radius $r$.
\end{Remark}

\subsection{Reverse isoperimetric problems for \texorpdfstring{$r$}{r}-ball bodies} 

The classical isoperimetric inequality (\cite{Ro}) implies in a straightforward way that among $r$-ball bodies of given volume the ball has the smallest surface area in $\Ee^d$. On the other hand, Borisenko (\cite{CDT}) conjectured that among $r$-ball bodies of given $d$-dimensional volume in $\Ee^d$, $d>1$ the $r$-lens, which is the intersection of two balls of radius $r$,  is the unique maximizer of surface area. {\it Borisenko's conjecture} was proved for $d=2$ by  Borisenko and Drach in \cite{BD14} and very recently for $d=3$ by Drach and Tatarko in \cite{DT23}. Thus, we have the following reverse isoperimetric theorem for $r$-ball bodies in $\Ee^d$ for $d=2,3$. 
\begin{Theorem}\label{BDT-theorem}
Let $r>0$. Among $r$-ball bodies (resp., $r$-disk domains) of given volume (resp., area) in $\Ee^3$ (resp., $\Ee^2$), the $r$-lens is the only one whose surface area (resp., perimeter) is maximal.
\end{Theorem}
As an extension of Borisenko's conjecture, the first named author (\cite{Bez23}) suggested to investigate the following reverse isoperimetric problems for $r$ ball bodies in $\Ee^d$, which seem to be open except for $k=d-1$ with $d=2,3$.

\begin{Problem}\label{reverse IP-problem-1}
Let $0<k<d$ and $r>0$. Among $r$-ball bodies of given $d$-dimensional volume in $\Ee^d$, $d>1$ find the one(s), whose $k$-th intrinsic volume is the largest possible.   
\end{Problem}
}
The recent paper \cite{Be21} proves the following - closely related - reverse inradius inequality on the $d$-dimensional volume of $r$-ball bodies in $\Ee^d$.  Let $\KK$ be a convex body in $\Ee^d$. Then its {\it inradius} $r_{in}(\KK)$ (resp., {\it circumradius} $r_{cr}(\KK)$) is the radius of the largest (resp., smallest) ball contained in $\KK$ (resp., containing $\KK$). In particular, we use the notation $L_{r, \rho, d}$ for the $r$-lens whose inradius is $\rho$ in $\Ee^d$, where $r\geq\rho>0$.

\begin{Theorem}\label{max-volume-lense}
Let $r>r_0>0$, $N>1, d>1$, and let $P:=\{\mathbf{p}_1,\mathbf{p}_2,\dots ,\mathbf{p}_N\}\subset\Ee^d$ with $r_{cr}(P)=r_0$. Then
\begin{equation}\label{Bezdek-1}
 \ivol[d]{P^r}\leq  \ivol[d]{L_{r, r-r_0, d}}.
\end{equation}
\end{Theorem} 

\begin{Remark}\label{equivalent-lense}
We note that $r_{in}(P^r)=r-r_0$ in Theorem~\ref{max-volume-lense}. Thus, the proof of Theorem~\ref{max-volume-lense} in \cite{Be21} implies the somewhat stronger statement that among $r$-ball bodies of given inradius in $\Ee^d$, the $r$-lens is the only one whose $d$-dimensional volume is maximal.
\end{Remark}

Moreover, notice that Theorem~\ref{BDT-theorem} and Theorem~\ref{max-volume-lense} yield the following statement in a straightforward way. Let $r>0$. Among $r$-ball bodies (resp., $r$-disk domains) of given inradius in $\Ee^3$ (resp., $\Ee^2$), the $r$-lens is the only one whose surface area (resp., perimeter) is maximal. Very recently Drach and Tatarko \cite{DT23} extended this result to higher dimensions.

\begin{Theorem}\label{BBDT-theorem}
Let $r>0$. Among $r$-ball bodies of given inradius in $\Ee^d$, $d>1$, the $r$-lens is the only one whose surface area is maximal.
\end{Theorem}

Theorems~\ref{max-volume-lense} and~\ref{BBDT-theorem} support the following conjecture of the first named author \cite{Be21}:

\begin{Conjecture}\label{max intrinsic-vol under inradius}
Let $0<k\leq d-2$ and $r>0$.  Among $r$-ball bodies of given inradius in $\Ee^d$, $d>1$, the $r$-lens is the only one whose $k$-th intrinsic volume is maximal.
\end{Conjecture}

We close this section with the following reverse circumradius inequality on the $d$-dimensional volume of $r$-ball bodies in $\Ee^d$, which was proved in \cite{Be21}. We use the notation $S_{r, \lambda, d}$ for the $r$-spindle whose circumradius is $\lambda$ in $\Ee^d$.

\begin{Theorem}\label{min-volume-spindle}
Let $r>r_0>0$, $N>1, d>1$, and let $P:=\{\mathbf{p}_1,\mathbf{p}_2,\dots ,\mathbf{p}_N\}\subset\Ee^d$ with $r_{cr}(P)=r_0$. Then
\begin{equation}\label{Bezdek-4}
V_d(S_{r, r_0, d})\leq V_d\left({\rm conv}_rP\right).
\end{equation}
\end{Theorem} 

\begin{Remark}\label{equivalent-spindle}
Theorem~\ref{min-volume-spindle} is equivalent to the statement that among $r$-ball bodies of given volume in $\Ee^d$, the $r$-spindle has the largest circumradius.
\end{Remark}

Theorem~\ref{min-volume-spindle} supports the following conjecture of the first named author \cite{Be21}:

\begin{Conjecture}\label{min-intrinsic-volume-spindle}
Let $r>r_0>0$, $N>1, d>k>0$, and let $P:=\{\mathbf{p}_1,\mathbf{p}_2,\dots ,\mathbf{p}_N\}\subset\Ee^d$ with $r_{cr}(P)=r_0$. Then
\begin{equation}\label{Bezdek-7}
V_k(S_{r, r_0, d})\leq V_k\left({\rm conv}_rP\right).
\end{equation}
\end{Conjecture}

\section{Spindle convex sets}\label{sec:spindle}

\subsection{Approximation of spindle convex sets}

Approximating convex sets by convex polytopes has been a longstanding problem of convex geometry. In these problems we consider a family $\FF$ of convex polytopes and a convex body $\KK$ in $\Ed$, and try to find the elements of $\FF$ with `minimal distance' from $\KK$. For example, $\FF$ may contain all convex polytopes with a given number of vertices or facets, where we may restrict our investigation to polytopes contained in $\KK$ or to polytopes containing $\KK$, and we can measure `distance' by various geometric quantities like the Hausdorff distance of $\KK$ and the polytope, or the volume of their symmetric difference. In some cases, instead of the `best' approximation we intend to find the distance of a `typical' approximation, and measure the average distance of $\KK$ from a randomly chosen element of $\FF$
using some fixed probability distribution on $\KK$. The latter problem is called the problem of \emph{random approximation} of a convex body. The systematic investigation for best approximation was started in the paper \cite{MCV1975} of McClure and Vitale, whereas for random approximation in the papers \cites{RS1963, RS1964, RS1968} of R\'enyi and Sulanke. For more information on this topic, the interested reader is referred to the recent survey \cite{PSW2022}.

In this section we consider a spindle convex variant of this problem (see the definition of spindle convexity right before Proposition~\ref{equivalence}). In this setting, the `counterpart' of a convex polytope can be either the spindle convex hull of finitely many points, called a \emph{ball-polytope} \cite{BLNP}*{Definition 13.1}, or the intersection of finitely many unit balls, called a \emph{ball-polyhedron}. These two concepts coincide in $\Ee^2$, and are called \emph{circle-polygons} or \emph{disk-polygons}. Our aim is to investigate how a spindle convex body can be approximated by ball-polytopes or ball-polyhedra.

With this aim, we introduce three metrics commonly used to measure the error of the approximation.

\begin{Definition}\label{defn:deviations}
Let $\KK$ and $\LL$ be convex bodies in $\Ed$. Then the \emph{volume} and \emph{surface volume deviations} of $\KK$ and $\LL$ are defined as
\[
\delta_V (\KK,\LL) = \ivol[d]{\KK \cup \LL} - \ivol[d]{\KK \cap \LL}, \quad \delta_S (\KK,\LL) = \surf(\KK \cup \LL) - \surf(\KK \cap \LL),
\]
respectively, where $\surf(\cdot)$ denotes surface volume. In addition, $\delta_H(\KK,\LL)$ denotes the Hausdorff distance of $\KK$ and $\LL$.
\end{Definition}

In the case of plane convex bodies, which we also call \emph{convex disks}, volume and surface volume deviation is usually called \emph{area} and \emph{perimeter deviation}, respectively.

The first result in this area, due to Fodor and V\'\i gh \cite{FV2012} gives asymptotic formulas for the deviations of the best approximating inscribed and circumscribed disk-$n$-gons. To formulate this results, let $\KK$ be a spindle convex body with $C^2$-class boundary, where for any $\xx \in \bd (\KK)$, $\kappa(\xx)$ denotes the curvature of $\KK$ at $\xx$. Choose disk-polygons $\KK_n^V$, $\KK_n^S, \KK_n^H$, with at most $n$ sides and inscribed in $\KK$, and closest to it with respect to area and perimeter deviation, and with respect to Hausdorff distance, respectively. We choose best approximating disk polygons $\KK_{(n)}^V, \KK_{(n)}^S, \KK_{(n)}^H$ similarly in the family of disk-polygons with at most $n$ sides, circumscribed about $\KK$. Then the following asymptotic formulas hold, where $d \xx$ denotes integration with respect to Hausdorff measure.

\begin{Theorem}\label{thm:FV2012}
For any spindle convex body $\KK$ with twice continuously differentiable boundary, we have
\begin{align*}
\lim_{n \to \infty} \left( n^2 \delta_V(\KK, \KK_n^V) \right)  &= \frac{1}{12} \left( \int_{\bd(\KK)} \left( \kappa(\xx)-1 \right)^{1/3} \, d \xx  \right)^3,\\
\lim_{n \to \infty} \left( n^2 \delta_{S}(\KK, \KK_n^{S}) \right) &= \frac{1}{24} \left( \int_{\bd(\KK)} \left( \kappa^2(\xx) -1 \right)^{1/3} \, d \xx \right)^3,\\
\lim_{n \to \infty} \left( n^2 \delta_H(\KK, \KK_n^H) \right) &= \frac{1}{8} \left( \int_{\bd(\KK)} \left( \kappa(\xx) -1 \right)^{1/2} \, d \xx \right)^2, \\
\lim_{n \to \infty} \left( n^2 \delta_V(\KK, \KK_{(n)}^V) \right) &= \frac{1}{24} \left( \int_{\bd(\KK)} \left( \kappa(\xx)-1 \right)^{1/3} \, d \xx  \right)^3,\\
\lim_{n \to \infty} \left( n^2 \delta_{S}(\KK, \KK_{(n)}^{S}) \right) &= \frac{1}{24} \left( \int_{\bd(\KK)} \left( 2 \kappa^2(\xx) - 3 \kappa(\xx) + 1 \right)^{1/3} \, d \xx \right)^3, \\
\lim_{n \to \infty} \left( n^2 \delta_H(\KK, \KK_{(n)}^H) \right) &= \frac{1}{8} \left( \int_{\bd(\KK)} \left( \kappa(\xx) -1 \right)^{1/2} \, d \xx \right)^2 .
\end{align*}
\end{Theorem}

In a very recent paper, Nagy and V\'\i gh \cite{NV2024} generalized this result for so-called $\LL$-convex sets, introduced by Mayer \cite{Ma} and systematically investigated, independently, by L\'angi et al. \cite{LNT} and Polovinkin \cite{P2000}. 
To state their result, first we introduce this concept. Let $\LL$ be a convex body in $\Ed$. A set $\KK$ is said to be \emph{$\LL$-ball-convex}, if it is equal to the intersection of all translates of $\LL$ that contain $\KK$. 
If $X \subset \Ed$ is contained in a translate of $\LL$, the $\LL$-convex hull of $X$, denoted by $\conv_{\LL} (X)$, is defined as the intersection of all translates of $\LL$ that contain $X$. In particular, if 
$\pp,\qqq \in \Ed$ are contained in a translate of $\LL$, then their $\LL$-convex hull is called the $\LL$-spindle of 
$\pp, \qqq$, denoted as 
$[\pp,\qqq]_{\LL}$. 
A set $\KK \subset \Ed$ is called \emph{$\LL$-spindle-convex}, if it is contained in a translate of $\LL$, and for any 
$\pp, \qqq \in \KK$, we have $[\pp, \qqq]_{\LL} \subseteq \KK$. Clearly, every $\LL$-ball-convex set is $\LL$-spindle-convex. In the planar case the converse also holds, but for any $d > 2$ there are examples of $\LL$-spindle convex sets that are not $\LL$-ball-convex for some suitable convex body $\LL \subset \Ed$ (see e.g. \cite{LNT}). On the other hand, there are also bodies $\LL$ different from ellipsoids for which $\LL$-spindle and $\LL$-ball convex sets coincide, as the example of a cube shows.
In the following we restrict ourselves to the planar case, and call an $\LL$-ball-convex (or equivalently, an $\LL$-spindle-convex) set shortly $\LL$-convex.

In $\Ee^2$, the intersection of at most $n$ translates of $\LL$ is called an $\LL$-$n$-gon. This property is equivalent to the one that the set is the $\LL$-convex hull of at most $n$ points in $\Ee^2$, contained in a translate of $\LL$ \cite{BL2024}. Now, let $\KK$ be an $\LL$-convex disk in $\Ee^2$. Choose $\LL$-$n$-gons $\KK_n^V$, $\KK_n^S, \KK_n^H$, contained in $\KK$, and closest to it with respect to area and perimeter deviation, and with respect to Hausdorff distance, respectively. We choose best approximating $\LL$-$n$-gons $\KK_{(n)}^V, \KK_{(n)}^S, \KK_{(n)}^H$ similarly in the family of $\LL$-$n$-gons, containing $\KK$. To state the corresponding result in \cite{NV2024}, we assume that both $\LL$ and $\KK$ have $C^2_+$-class boundary, that is, $\bd (\LL)$ and $\bd(\KK)$ are twice continuously differentiable, with strictly positive curvature everywhere. We denote by $r_{\LL}(\theta)$ (resp. by $r_{\KK}(\theta)$) the radius of curvature of $\bd(\LL)$ (resp. $\bd (\KK)$) at the unique point with outer unit normal vector $(\cos \theta, \sin \theta)$. We also note that the assumption that $\KK$ is $\LL$-convex is equivalent to the property that $r_{\KK}(\theta) \leq r_{\LL}(\theta)$ for every value of $\theta$. Then the following asymptotic formulas hold, which generalize the result in Theorem~\ref{thm:FV2012}.

\begin{Theorem}\label{thm:NV2024}
For convex disk $\LL \subset \Ee^2$ with $C^2_+$-class boundary, and for any $\LL$-convex disk $\KK$ with $C^2_+$-class boundary:
\begin{align*}
\lim_{n \to \infty} \left( n^2 \cdot \delta_V(\KK, \KK_n^V) \right)  &= \frac{1}{12} \left( \int_0^{2\pi} \left( \frac{r_{\KK}^2(\theta)}{r_{\LL}(\theta)} \left( r_{\LL}(\theta) - r_{\KK}(\theta) \right)\right)^{1/3} \, d \theta  \right)^3,\\
\lim_{n \to \infty} \left( n^2 \cdot \delta_A(\KK, \KK_n^A) \right)  &= \frac{1}{24} \left( \int_0^{2\pi} \left( \frac{r_{\KK}(\theta)}{r_{\LL}^2(\theta)} \left( r_{\LL}^2(\theta) - r_{\KK}^2(\theta) \right)\right)^{1/3} \, d \theta  \right)^3,\\
\lim_{n \to \infty} \left( n^2 \cdot \delta_H(\KK, \KK_n^H) \right)  &= \frac{1}{8} \left( \int_0^{2\pi} \left( \frac{r_{\KK}(\theta)}{r_{\LL}(\theta)} \left( r_{\LL}(\theta) - r_{\KK}(\theta) \right)\right)^{1/2} \, d \theta  \right)^2.
\end{align*}
Furthermore, if, in addition, $\LL$ is centrally symmetric, then
\begin{align*}
\lim_{n \to \infty} \left( n^2 \cdot \delta_V(\KK, \KK_{(n)}^V) \right)  &= \frac{1}{24} \left( \int_0^{2\pi} \left( \frac{r_{\KK}^2(\theta)}{r_{\LL}(\theta)} \left( r_{\LL}(\theta) - r_{\KK}(\theta) \right)\right)^{1/3} \, d \theta  \right)^3,\\
\lim_{n \to \infty} \left( n^2 \cdot \delta_A(\KK, \KK_{(n)}^A) \right)  &= \frac{1}{24} \left( \int_0^{2\pi} \left( \frac{r_{\KK}^3(\theta) - 3 r_{\KK}^2(\theta) r_{\LL}(\theta) + 2 r_{\KK}(\theta) r_{\LL}^2(\theta)}{r_{\LL}^2(\theta)} \right)^{1/3} \, d \theta  \right)^3,\\
\lim_{n \to \infty} \left( n^2 \cdot \delta_H(\KK, \KK_{(n)}^H) \right)  &= \frac{1}{8} \left( \int_0^{2\pi} \left( \frac{r_{\KK}(\theta)}{r_{\LL}(\theta)} \left( r_{\LL}(\theta) - r_{\KK}(\theta) \right)\right)^{1/2} \, d \theta  \right)^2.\\
\end{align*}
\end{Theorem}

\subsection{Random disk-polygons}
Consider the following probability model: Let $\KK$ be a convex disk in the plane $\Ee^2$. Choose points $\pp_1,\pp_2, \ldots, \pp_n$ independently from $\KK$, using the uniform distribution. Set $\KK_n = \conv \{ \pp_1, \pp_2, \ldots, \pp_n \}$. We call $\KK_n$ a \emph{random convex $n$-gon}. Then it is a natural question to ask about the expected value of some geometric quantities of $\KK_n$, for instance the number of its vertices, or its area or perimeter deviation from $\KK$. 
These questions were investigated in the classical papers \cites{RS1963, RS1964, RS1968} of R\'enyi and Sulanke, under some smoothness assumption on $\KK$. We note that, unknown by R\'enyi and Sulanke at that time, there is a straighforward connection between the number of vertices of $\KK_n$ and its area deviation from $\KK$, known as \emph{Efron's formula} \cite{E1965}.

The above problem has a natural analogue in case of spindle convex sets: Let $\KK$ be a spindle convex disk, and choose points $\pp_1,\pp_2, \ldots, \pp_n$ independently from $\KK$, using the uniform distribution. Define $\KK_n$ as the spindle convex hull of $\pp_1,\pp_2, \ldots, \pp_n$. Then we can ask about the expected value of the number of vertices of $\KK_n$, and its area or perimeter deviation from $\KK$. These questions were investigated by Fodor, Kevei and V\'\i gh in \cite{FKV2014}. We state their results as follows; where $\Ee(\cdot)$ denotes expected value, $f_0(\KK_n)$ denotes the number of vertices of $\KK_n$, and for every $\xx \in \bd(\KK)$, $\kappa(\xx)$ denotes the curvature of $\bd(\KK)$ at $\xx$.

\begin{Theorem}\label{thm:FKV2014_1}
Let $\KK$ be a spindle convex disk with $C^2$-class boundary and with $\kappa(\xx) > 1$ for every $\xx \in \bd(\KK)$. Then
\begin{align}
\label{eq:FKV_1a} \lim_{n \to \infty} \left( n^{-1/3} \cdot \Ee(f_0(\KK_n)) \right)  &= \sqrt[3]{\frac{2}{3 \ivol[2]{\KK}}} \Gamma \left( \frac{5}{3} \right) \int_{\bd(\KK)} \left( \kappa(\xx) - 1  \right)^{1/3} \, d \xx,\\
\label{eq:FKV_1b} \lim_{n \to \infty} \left( n^{2/3} \cdot \Ee(\ivol[2]{\KK \setminus \KK_n} ) \right) &= \sqrt[3]{\frac{2 \ivol[2]{\KK}^2 }{3}} \Gamma \left( \frac{5}{3} \right) \int_{\bd(\KK)} \left( \kappa(\xx) - 1  \right)^{1/3} \, d \xx
\end{align}
Furthermore, if $\KK$ has $C^5$-class boundary, with $\kappa(\xx) > 1$ for every $\xx \in \bd(\KK)$, then
\begin{multline}
\lim_{n \to \infty} \left( n^{2/3} \cdot \Ee( \perim (\KK) - \perim (\KK_n) ) \right) = \\
\frac{\sqrt[3]{144 \ivol[2]{\KK}^2}}{36} \Gamma \left( \frac{2}{3} \right) \int_{\bd(\KK)} \left( \kappa(\xx) - 1  \right)^{1/3} \left( 3 \kappa(\xx) + 1  \right) \, d \xx.
\end{multline}
In the above formulas, the integration is with respect to arclength i.e. $1$-dimensional Hausdorff measure.
\end{Theorem}
 
It is worth noting that like in the original problem for linearly convex polygons, (\ref{eq:FKV_1a}) is readily implied by (\ref{eq:FKV_1b}) by an Efron-type formula, proved in \cite{FKV2014}, stating that
\begin{equation}\label{eq:spEfron}
\Ee(f_0(\KK_n)) = \frac{n \Ee(\ivol[2]{\KK \setminus \KK_{n-1}} )}{\ivol[2]{\KK}}. 
\end{equation}
Note the necessary condition that the curvature of $\KK$ is strictly greater than $1$, that is the curvature of the unit circle. If we omit this condition, the order of magnitude of the above quantities may change. This is illustrated in the next theorem from \cite{FKV2014} for the special case that $\KK$ is the Euclidean unit disk.

\begin{Theorem}\label{thm:FKV2014_2}
If $\KK = \mathbf{B}^2[\oo, 1 ]$, then 
\begin{align}
\label{eq:diskvertex} \lim_{n \to \infty} \Ee(f_0(\KK_n))  &=  \frac{\pi^2}{2},\\ 
\lim_{n \to \infty} \left( n \cdot \Ee(\ivol[2]{\KK \setminus \KK_n} ) \right) &= \frac{\pi^3}{2},\\
\lim_{n \to \infty} \left( n \cdot \Ee(\perim (\KK) - \perim (\KK_n) ) \right) &= \frac{\pi^3}{2}.
\end{align}
\end{Theorem}

We remark that the interesting phenomenon in (\ref{eq:diskvertex}), namely that the expected value of the number of sides of $\KK_n$ is a constant independent of $n$, does not appear in the original problem about the number of sides of linearly convex polygons. Later, a similar phenomenon was observed for the number of facets of certain random polytopes in spherical space (see \cite{BHR2017}).

Similar to the previous result, we can investigate the asymptotic values of the variances $\var (\cdot)$ of the same quantities. For the case of the number of vertices and for area deviation, this was examined in \cite{FV2018} by Fodor and V\'\i gh. We present their results as follows.

\begin{Theorem}\label{thm:FV2018_1}
Let $\KK$ be a spindle convex disk with $C^2$-class boundary and with $\kappa(\xx) > 1$ for every $\xx \in \bd(\KK)$. Then
\begin{align}
\var(f_0(\KK_n))  &= \OO\left( n^{1/3}\right)\\
\var(\ivol[2]{\KK_n}) &= \OO\left( n^{-5/3}\right),
\end{align}
where the implied constants depend only on $\KK$.
\end{Theorem}

\begin{Theorem}\label{thm:FV2018_2}
If $\KK = \mathbf{B}^2[\oo, 1 ]$, then 
\begin{align}
\var(f_0(\KK_n))  &= \Theta\left( 1 \right)\\
\var(\ivol[2]{\KK_n}) &= \OO\left( n^{-2}\right).
\end{align}
\end{Theorem}

It is worth noting that Theorems \ref{thm:FKV2014_1}-\ref{thm:FV2018_2}, combined with Chebyshev's inequality and the Borel-Cantelli lemma, imply the following statement \cite{FV2018}.

\begin{Theorem}\label{thm:FV2018_3}
If $\KK$ is a spindle convex disk with $C^2$-class boundary and with $\kappa(\xx) > 1$ for every $\xx \in \bd(\KK)$, then it holds with probability $1$ that
\begin{align}
\label{eq:FV_3a} \lim_{n \to \infty} \left( n^{-1/3} \cdot f_0(\KK_n) \right)  &= \sqrt[3]{\frac{2}{3 \ivol[2]{\KK}}} \Gamma \left( \frac{5}{3} \right) \int_{\bd(\KK)} \left( \kappa(\xx) - 1  \right)^{1/3} \, d \xx,\\
\label{eq:FV_3b} \lim_{n \to \infty} \left( n^{2/3} \cdot \ivol[2]{\KK \setminus \KK_n} \right) &= \sqrt[3]{\frac{2 \ivol[2]{\KK}^2 }{3}} \Gamma \left( \frac{5}{3} \right) \int_{\bd(\KK)} \left( \kappa(\xx) - 1  \right)^{1/3} \, d \xx.
\end{align}
\end{Theorem}

We remark that in \cite{FGV2022}, Fodor, Gr\"unfelder and V\'\i gh proved that the upper bounds on the variances in Theorem~\ref{thm:FV2018_1} are asymptotically tight, by proving matching lower bounds on these quantities.

Following the literature about approximations of convex bodies by polytopes, one can aim at finding power-series expansions of the expected values of the quantities investigated above. Regarding it, Fodor and Montenegro Pinz\'on \cite{FP2024} proved the following.

\begin{Theorem}\label{thm:FP2024_1}
Let $\KK$ be a spindle convex disk with $C^{k+1}$-class boundary and with $\kappa(\xx) > 1$ for every $\xx \in \bd(\KK)$. Then
\[
\Ee \left( f_0(\KK_n) \right) = z_1(\KK) n^{1/3} + \ldots + z_{k-1}(\KK_n) n^{-(k-3)/3} + \OO\left( n^{-(k-2)/3} \right)
\]
as $n \to \infty$. All the coefficients $z_1(\KK), \ldots, z_{k-1}(\KK)$ can be computed explicitely. In particular,
\begin{align}
z_1(\KK)  &= \sqrt[3]{\frac{2}{3 \ivol[2]{\KK}}} \Gamma \left( \frac{5}{3} \right) \int_{\bd(\KK)} \left( \kappa(\xx) - 1  \right)^{1/3} \, d \xx.\\
z_2(\KK) &= 0.
\end{align}
\begin{multline}
z_3(\KK) = - \Gamma \left( \frac{7}{3} \right) \sqrt[3]{\frac{3 \ivol[2]{\KK}}{2}} \cdot\\
\cdot \int_{\bd(\KK)} \left( \frac{\kappa''(\xx)}{3 \left( \kappa(\xx) - 1 \right)^{4/3}} + \frac{2\kappa^2(\xx) + 7\kappa(\xx) -1}{2 \left( \kappa(\xx) - 1 \right)^{4/3}}- - \frac{5\left( \kappa'(\xx) \right)^2 }{9 \left( \kappa(\xx) - 1 \right)^{7/3}} \right) \, d \xx.
\end{multline}
\end{Theorem}

Combining Theorem~\ref{thm:FP2024_1} with the spindle convex Efron's formula (\ref{eq:spEfron}), we readily obtain Theorem~\ref{thm:FP2024_2}.

\begin{Theorem}\label{thm:FP2024_2}
Let $\KK$ be a spindle convex disk with $C^{k+1}$-class boundary and with $\kappa(\xx) > 1$ for every $\xx \in \bd(\KK)$. Then
\[
\Ee \left( \ivol[2]{\KK \setminus \KK_n} \right) = z_1'(\KK) n^{1/3} + \ldots + z_{k-1}'(\KK_n) n^{-(k-3)/3} + \OO\left( n^{-(k-2)/3} \right)
\]
as $n \to \infty$, where $z_i'(\KK) = \ivol[2]{\KK} z_i(\KK)$.
\end{Theorem}

Unlike in the case of best approximation, for random approximation it is a meaningful problem to investigate the deviation between a convex polygon and a random convex $n$-gon approximating it, as $n \to \infty$. This problem was examined in the classical papers of R\'enyi and Sulanke \cites{RS1963, RS1964}. The spindle convex variant of this question was studied by Fodor, Kevei and V\'\i gh \cite{FKV2023}, who proved the following.

\begin{Theorem}\label{thm:FKV2023}
Let $\KK$ be a disk polygon different from $\mathbf{B}^2[\oo, 1 ]$. Then
\[
\lim_{n \to \infty} \frac{\Ee \left( f_0(\KK_n) \right)}{\ln n} = \frac{2}{3} f_o(\KK).
\]
Furthermore,
\[
\lim_{n \to \infty} \frac{n \Ee \left( V_2(\KK \setminus \KK_n) \right)}{\ln n} = \frac{2}{3} f_o(\KK) \ivol[2]{\KK}.
\]
\end{Theorem}

Finally, we note that similar results related to spindle convex sets have appeared in the recent paper \cite{FPap2024}.

\subsection{\texorpdfstring{$\LL$}{L}-convex sets}
Random approximation is also investigated for $\LL$-convex sets. In particular, Fodor, Papv\'ari and V\'\i gh \cite{FPV2020} studied the following model: Let $\LL$ be a convex disk in $\Ee^2$, and let $\KK$ be an $\LL$-convex disk. Let $\KK_n$ denote the $\LL$-convex hull of $n$ random points $\xx_1, \xx_2,\ldots, \xx_n$ of $\KK$ chosen independently from $\KK$ using the uniform distribution. Then $\KK_n$ is an $\LL$-$m$-gon for some $m \leq n$, and if $\LL$ is strictly convex and smooth, then $m$ coincides with the number of elements of $\{ \xx_1, \xx_2,\ldots, \xx_n \}$ lying in the boundary of $\KK_n$, we denote the value $m$ as $f_0(\KK_n)$. The authors of \cite{FPV2020} examined the properties of $\KK_n$ in two special cases. In the first one, they assumed that both $\LL$ and $\KK$ have $C^2$-class boundaries with strictly positive curvature. Furthermore, if, for any vector $\uu \in \Sph^1$, $\kappa_{\LL}(\uu)$ (resp. $\kappa_{\KK}(\uu))$ denotes the curvature of $\bd (\LL)$ (resp. $\bd(\KK)$) at the point with outer normal vector $\uu$, they assumed that
\[
\max \{ \kappa_{\LL}(\uu) : \uu \in \Sph^1 \} < 1 < \min \{ \kappa_{\KK}(\uu) : \uu \in \Sph^1 \}.
\]

\begin{Theorem}\label{thm:FPV2020_1}
With the above conditions, it follows that
\[
\lim_{n \to \infty} \left( n^{-1/3} \Ee (f_0(\KK_n))\right) = \sqrt[3]{\frac{2}{3 \ivol[2]{\KK}}} \Gamma \left( \frac{5}{3} \right) \int_{\Sph^1} \frac{\left( \kappa_{\KK}(\uu) - \kappa_{\LL}(\uu)  \right)^{1/3}}{\kappa_{\LL}(\uu) } \, d \uu.
\]
\end{Theorem}

The $\LL$-convex variant of Efron's identity states that
\begin{equation}\label{eq:LEfron}
\Ee (f_0(\KK_n)) = n \frac{1}{\ivol[2]{\KK}} \Ee(\ivol[2]{\KK \setminus \KK_{n-1}}).
\end{equation}
This formula, combined with Theorem~\ref{thm:FPV2020_1} yields Corollary~\ref{cor:FPV}.

\begin{Corollary}\label{cor:FPV}
With the conditions of Theorem~\ref{thm:FPV2020_1}, it follows that
\[
\lim_{n \to \infty} \left( n^{-1/3} \Ee(\ivol[2]{\KK \setminus \KK_n})\right) = \sqrt[3]{\frac{2 (\ivol[2]{\KK})^2 }{3 }} \Gamma \left( \frac{5}{3} \right) \int_{\Sph^1} \frac{\left( \kappa_{\KK}(\uu) - \kappa_{\LL}(\uu)  \right)^{1/3}}{\kappa_{\LL}(\uu) } \, d \uu.
\]
\end{Corollary}

The other special case, investigated in \cite{FPV2020}, is that $\KK = \LL$. We present a corrected version of their results\footnote{The corrected formulas were provided for us by the authors of \cite{FPV2020} } . To state it, we introduce some notation. For any $\uu \in \Sph^2$ and convex disk $\LL$, we denote by $w_{\LL}(\uu)$ the width of $\LL$ in the direction perpendicular to $\uu$.
Furthermore, for any semicircle $L^*(\uu)\subset \Sph^1$ with center $\uu$, we define
\begin{align*}
I^*(\uu)=\int_{L^*(\uu)}\int_{L^*(\uu)}\frac{\left|\uu_1\times \uu_2\right|}{\kappa_L(\uu_1)\kappa_L(\uu_2)} \, d \uu_1 \, d \uu_2.
\end{align*}
Then the following hold.

\begin{Theorem}\label{thm:FPV2020_2}
If $\LL$ is a convex disk with $C^2$-class boundary and strictly positive curvature, then
\begin{align*}
	 \lim_{n\to\infty}\Ee (f_0(\LL_{n}))&=\frac 12\int_{\Sph^1} \frac{I^*(\uu)}{w^2(\uu)} \, d \uu,\\
	\lim_{n\to\infty} \left( n \Ee (\ivol[2]{\LL\setminus \LL_{n}}) \right) &=\frac {\ivol[2]{\LL}}{2}\int_{\Sph^1} \frac{I^*(\uu)}{w^2(\uu)} \, d \uu .
	\end{align*} 
\end{Theorem}

A similar problem was investigated in \cite{NV2024} by Nagy and V\'\i gh. To present their result, consider a convex disk $\LL$ with $C^2$-class boundary and strictly positive curvature, and an $\LL$-convex disk $\KK$ satisfying the same smoothness properties. Let $\mu : \bd (\KK) \to \Re$ be an arbitrary, continuous density function on $\bd (\KK)$. Choose $n$ points $\xx_1, \xx_2, \ldots, \xx_n$ of $\bd(\KK)$ independently according to the probability distribution defined by $\mu$. Let $\KK_n$ denote the $\LL$-convex hull of these points. Furthermore, for any $\xx_i$, let $\LL_i$ denote the unique translate of $\LL$ whose boundary touches $\KK$ at $\xx_i$, and satisfies $\KK \subseteq \LL_i$. We define $\KK_{(n)}= \cap_{i=1}^n \LL_i$. Note that both $\KK_n$ and $\KK_{(n)}$ are random $\LL$-gons, where the first one is inscribed in $\KK$, and the second one is circumscribed about $\KK$. One of the goals of \cite{NV2024} is to investigate the geometric properties of $\KK_n$ and $\KK_{(n)}$.
 
To do it, for any $\theta \in [0,2\pi]$, let $\xx_{\LL}(\theta)$, $\kappa_{\LL}(\theta)$ and $r_{\LL}(\theta)$ denote the point of $\bd(\LL)$ with outer unit normal vector $(\cos \theta, \sin \theta )$, and the curvature and radius of curvature of $\bd(\LL)$ at $\xx_{\LL}(\theta)$, respectively. We define $\xx_{\KK}(\theta)$, $\kappa_{\KK}(\theta)$ and $r_{\KK}(\theta)$ similarly for $\KK$. Finally, we set $m(\theta) = \frac{\mu(\xx_{\KK}(\theta))}{\kappa_{\KK}(\theta)}$ for all values of $\theta$. The main results of \cite{NV2024} related to the above problem are the following.

\begin{Theorem}\label{thm:NV2024_1}
Let $\KK$ and $\LL$ be convex disks with $C^2$-class boundary and strictly positive curvature. Then, with probability $1$, the following holds:
\begin{align}
\lim_{n \to \infty} \left( n^2 \delta_A(\KK, \KK_n) \right) &= \frac{1}{4} \int_0^{2\pi} \frac{r_{\KK}(\theta)}{r_{\LL}^2(\theta)} \left( r_{\LL}^2(\theta) - r_{\KK}^2(\theta) \right) m^{-2}(\theta) \, d \theta, \\
\lim_{n \to \infty} \left( n^2 \delta_V(\KK, \KK_n) \right) &= \frac{1}{2}  \int_0^{2\pi} \frac{r_{\KK}^2(\theta)}{r_{\LL}(\theta)} \left( r_{\LL}(\theta) - r_{\KK}(\theta) \right) m^{-2}(\theta) \, d \theta.
\end{align}
Furthermore,
\begin{multline}
\lim_{n \to \infty} \left( \left( \frac{n}{\ln n} \right)^2 \delta_H(\KK, \KK_n) \right) = \\
\frac{1}{8}  \max \left\{ \frac{r_{\KK}(\theta)}{r_{\LL}(\theta)} \left( r_{\LL}(\theta) - r_{\KK}(\theta) \right) m^{-2}(\theta) : \theta \in [0,2\pi]  \right\}.
\end{multline}
\end{Theorem}

\begin{Theorem}\label{thm:NV2024_2}
Let $\KK$ and $\LL$ be convex disks with $C^2$-class boundary and strictly positive curvature. Assume that $\LL$ is centrally symmetric. Then, with probability $1$, the following holds:
\begin{align}
\lim_{n \to \infty} \left( n^2 \delta_A(\KK, \KK_{(n)}) \right) &= \frac{1}{4} \int_0^{2\pi} \frac{r_{\KK}^3(\theta) -3r_{\KK}^2(\theta) r_{\LL}(\theta) +2 r_{\KK}(\theta) r_{\LL}^2(\theta) }{r_{\LL}^2(\theta)} \, d \theta, \\
\lim_{n \to \infty} \left( n^2 \delta_V(\KK, \KK_{(n)}) \right) &= \frac{1}{4}  \int_0^{2\pi} \frac{r_{\KK}^2(\theta)}{r_{\LL}(\theta)} \left( r_{\LL}(\theta) - r_{\KK}(\theta) \right) m^{-2}(\theta) \, d \theta.
\end{align}
Furthermore,
\begin{multline}
\lim_{n \to \infty} \left( \left( \frac{n}{\ln n} \right)^2 \delta_H(\KK, \KK_{(n)} ) \right) = \\
\frac{1}{8}  \max \left\{ \frac{r_{\KK}(\theta)}{r_{\LL}(\theta)} \left( r_{\LL}(\theta) - r_{\KK}(\theta) \right) m^{-2}(\theta) : \theta \in [0,2\pi]  \right\}.
\end{multline}
\end{Theorem}

Similar problems were investigated by Marynich and Molchanov \cite{MM2022} in higher dimensions for $\LL$-ball-convex bodies. To state their results, we need a longer preparation, where, for simplicity, we assume that $\LL$ is strictly convex, smooth, contains the origin $\oo$ in its interior. Let $X \subset \LL$. Then the $\LL$-convex hull of $X$ can be obtained as
\begin{equation}\label{eq:Lhull}
\conv_{\LL} (X) = \cap_{\xx \in \Ed, X \subseteq \xx + \LL} (\xx + \LL) = \LL \ominus \left( \LL \ominus  X \right),  
\end{equation}
where $A \ominus B$ denotes the \emph{Minkowski difference} of $A, B \subseteq \Ed$, defined as $A \ominus B = \{ \xx \in \Ed : \xx + B \subseteq A \}$. We note if a set $B$ satisfies $A+C = B$ for some set $C$, we say that $B$ is a \emph{summand} of $A$, and if each intersection of its translates of $A$ is a summand of $A$, then $A$ is called a \emph{generating set} \cite{Sc14}. From now on, we assume also that $\LL$ is a generating set, which implies that a set $\KK$ is $\LL$-ball-convex if and only if it is $\LL$-spindle convex (see \cite{K2004}).

The definition of the faces of an $\LL$-convex body is highly nontrivial \cite{BLNP}. To do it, the authors of \cite{MM2022} define the faces of the convex hull of a (not necessarily finite) family $\mathcal{L}$ of compact convex sets. For simplicity, we present the result of their discussion only under the assumption that each element of $\mathcal{L}$ is strictly convex. Under this condition, and assuming that $\conv \mathcal{L}$ is compact, for any (exposed) proper face $F$ of $\conv \mathcal{L}$, we define the set $\mathcal{M}(\mathcal{L}, F) = \{ \KK \in \mathcal{L}: \KK \cap F \neq \emptyset \}$. In this case we have $F = \conv \left( \bigcup_{\KK \in \mathcal{M}(\mathcal{L}, F)} \left( \KK \cap F \right)  \right)$. Then we can say that $\mathcal{L}$ is in \emph{general position} if for any $k$-dimensional face $F$ of $\conv \mathcal{L}$, where $k=0,1,\ldots,d-1$, $\mathcal{M}(\mathcal{L}, F)$ consists of exactly $k+1$ elements of $\mathcal{L}$. In this case the family $\mathcal{M}_k(\mathcal{L})$ is defined as the set $\mathcal{M}_k(\mathcal{L}) = \left\{ \mathcal{M}(\mathcal{L}, F) : \dim F = k \right\}$, and we call the elements of this set the \emph{$k$-dimensional faces of the family $\mathcal{L}$}. We define the \emph{$\mathfrak{f}$-vector} of $\mathcal{L}$ as the vector
\[
\mathfrak{f}(\mathcal{L}) = \left( \mathfrak{f}_0(\mathcal{L}), \mathfrak{f}_1(\mathcal{L}), \ldots, \mathfrak{f}_{d-1}(\mathcal{L}) \right),
\]
where $\mathfrak{f}_k(\mathcal{L})$ denotes the cardinality of $\mathcal{M}_k(\mathcal{L})$ (counted without multiplicities). To define the $\mathfrak{f}$-vector of a set $A$ in the interior of $\LL$, the authors set $\mathcal{L}_A = \left\{ (\LL-\xx)^{\circ} : \xx \in A \right\}$, where $(\cdot)^{\circ}$ denotes Euclidean polar, and define the $\mathfrak{f}$-vector of $\mathbf{Q}= \conv_{\textcolor{red}{\LL}} (A)$ as
\[
\mathfrak{f}(\mathbf{Q}) = \mathfrak{f}(\mathcal{L}_A).
\]

In the following, choose $n$ random points $\xx_1, \ldots, \xx_n$ from $\LL$, independently and according to the uniform distribution on $\LL$ defined by $d$-dimensional Lebesgue measure. Let $\mathbf{Q}_n = \bigcap_{\xx \in \Ed, \Xi_n \subset \xx+\LL}$ denote the $\LL$-convex hull of $\Xi_n= \{ \xx_1, \xx_2, \ldots, \xx_n \}$. Let $\mathbf{X}_n = \LL \ominus \Xi_n$, and recall that $\mathbf{Q}_n = \LL \ominus \mathbf{X}_n$. By the properties of polarity, we have $\mathbf{X}_n^{\circ}= \conv \left( \bigcup_{i=1}^n \left( \LL - \xx_i \right)^{\circ} \right)$. The sets $\mathbf{X}_n$ and $\mathbf{X}_n^{\circ}$ are random convex bodies containing $\oo$ in their interiors, with probability $1$.

To state the first main result in \cite{MM2022}, for any unit vector $\uu \in \Sph^{d-1}$ and $t \in (0,\infty)$, we denote by $H^-_{\uu}(t)$ the half space $\{ \xx \in \Ed : \langle \xx, \uu \rangle \leq t \}$. Then, we define the Poisson process $\mathcal{P}_{\LL} = \{ (t_i,\uu_i) : i \geq 1\}$ on $(0,\infty) \times \Sph^{d-1}$, where the intensity measure $\mu$ of the process is the product of the Lebesgue measure on $(0,\infty)$ multiplied by $\frac{1}{\ivol[d]{\LL}}$, and the surface area measure $S_{d-1}(\LL,\cdot)$ on $\Sph^{d-1}$ (for more information on stochastic processes, the reader is referred to the book \cite{SW2008}).
Then the set $\{ H_{\uu_i}^-(t_i) : i \geq 1 \}$ is a collection of half spaces. The boundaries of these half spaces form a hyperplane process, forming a tiling of $\Ed$. The measure $S_{d-1}(\LL,\cdot)$ is called the \emph{directional component} of the tessellation. With probability one, the intersection $\mathbf{Z}$ of the half spaces $H_{\uu_i}^-(t_i)$ is a random convex polytope containing $\oo$ in its interior, called the \emph{zero cell} of the above hyperplane tessellation. The polar set $\mathbf{Z}^{\circ}$ of this set is the convex hull of the points $t_i^{-1} \uu_i$. The points form a Poisson point process $\Pi_{\LL}$, and their convex hull, with probability $1$, is a convex polytope containing $\oo$ in its interior.

\begin{Theorem}\label{thm:MM2022_1}
The sequence of the random convex bodies $\{ n \mathbf{X}_n \}$ converges in distribution, as $n \to \infty$ to the zero cell $\mathbf{Z}$ of the above hyperplane tessellation. Furthermore, the sequence of the random convex bodies $\left\{ n^{-1} \mathbf{X}_n^{\circ} \right\}$ converges in distribution to the random convex set $\mathbf{Z}^{\circ}$, as $n \to \infty$.
\end{Theorem}

\begin{Corollary}\label{cor:MM2022}
For any $j=0,1,\ldots, d$,
\[
\Ee \left( n^j \ivol[j]{\mathbf{X}_n} \right) \to \Ee \left( \ivol[j]{\mathbf{Z}} \right),
\]
as $n \to \infty$.
\end{Corollary}

Theorem~\ref{thm:MM2022_2} can be deduced from Theorem~\ref{thm:MM2022_1}.

\begin{Theorem}\label{thm:MM2022_2}
For any $j=0,1,\ldots,d-1$,
\[
\lim_{n \to \infty} \Ee \left( \mathfrak{f}_j (\mathbf{Q}_n) \right) = \Ee \left( \mathfrak{f}_j (\mathbf{Z}^{\circ}) \right) < \infty.
\]
\end{Theorem}

We note that the computation of the limit in the above theorem is extremely difficult. Nevertheless, if the directional distribution of the hyperplane process generating $\mathbf{Z}$ is even, then the number of vertices of $\mathbf{Z}$ is known \cite{S1982}. More specifically, if $\LL$ is centrally symmetric, then
\[
\lim_{n \to \infty} \Ee \left( \mathfrak{f}_{d-1} (\mathbf{Q}_n) \right) = \Ee \left( \mathfrak{f}_0 (\mathbf{Z}) \right) = 2^{-d} d! \ivol[d]{\Pi \LL} \ivol[d]{\Pi^{\circ} \LL}.
\]
Here, $\Pi \LL$ denotes the projection body of $\LL$, defined by the property that the value of its support function in the direction of $\uu \in \Sph^{d-1}$ coincides with the $(d-1)$-dimensional volume of its projection onto the orthogonal complement of $\uu$; furthermore $\Pi^{\circ} \LL$ denotes the polar of $\Pi \LL$.

Finally, we note that similar results related to $\LL$-convex sets have appeared in the recent paper \cite{FG2025}.

\subsection{Wendel's inequality}\label{subsec:Wendel}
Wendel's inequality \cite{W1962} states that if $\xx_1, \xx_2,\ldots, \xx_n$ are random points chosen from $\Ed$ independently according to an $\oo$-symmetric probability distribution, then the probability that their convex hull does not contain the origin is
\[
\mathbb{P}(\oo \notin \conv \{ \xx_1,\xx_2,\ldots, \xx_n \}) = \frac{1}{2^{n-1}} \sum_{i=0}^{d-1} \binom{n-1}{i} .
\]
Wagner and Welzl \cite{WW2001} extended this result by showing that if we drop the condition that the distribution is $\oo$-symmetric, then the $=$ sign in this formula should be replaced by the $\leq$ sign.

Fodor, Pinz\'on and V\'\i gh \cite{FPV2023} investigated a spindle convex variant of this problem in the following way.
Let $\KK$ be a spindle convex body in $\Ed$. Chose points $\xx_1, \xx_2,\ldots, \xx_n$ randomly and independently from $\KK$ according to the uniform distribution. 
Then we set
\[
P(d, \KK, n) = \mathbb{P} (\oo \notin \conv_s \{ \xx_1,\xx_2,\ldots, \xx_n \}),
\]
where $\conv_s(\cdot)$ denotes spindle convex hull.

\begin{Theorem}\label{thm:FPV2023_1}
We have
\[
P(d, r \BB^d, 2) = \frac{d(d+1) \omega_{d+1}}{r^{2d} \omega_d} \int_0^r \int_0^r \int_q^{\varphi(r_1,r_2)} r_1^{d-1} r_2^{d-1} \sin^{d-2} \varphi \, d \varphi \, d r_2 \, d r_1,
\]
where $\varphi(r_1,r_2) = \arcsin \frac{r_1}{2} + \arcsin \frac{r_2}{2}$.
\end{Theorem}

They also computed the value $P(2,\BB^2, 3)$, and investigated the case when two points are chosen according to the Gaussian distribution.

\subsection{Dowker's theorems}

For any integer $n \geq 3$ and plane convex body $\KK$, let $A_n(\KK)$ (resp. $a_n(\KK)$) denote the the infimum (resp. supremum) of the areas of the convex $n$-gons circumscribed about (resp. inscribed in) $\KK$. Verifying a conjecture of Kerschner, Dowker \cite{D1944} proved that for any plane convex body $\KK$, the sequences $\{ A_n(\KK) \}$ and $\{ a_n(\KK) \}$ are convex and concave, respectively. It was proved independently by L. Fejes T\'oth \cite{LFT1955}, Moln\'ar \cite{M1955} and Eggleston \cite{E1957} that the same statements remain true if we replace area by perimeter, where the last author also showed that these statements are false if we replace area by Hausdorff distance. These results are known to be true also in any normed plane \cite{MSW2001}, and in spherical and hyperbolic planes \cites{M1955, LFT1958}.

We collect the results related to the spindle convex variants of this problem. For this purpose, we introduce some notation; in the next definition $\perim_{\CC}(\cdot)$ denotes perimeter measured in the norm with unit disk $\CC$.

\begin{Definition}\label{defn:areper}
Let $n \geq 3$ and let $\KK$ be a $\CC$-convex disk in $\Ee^2$, where $\CC$ is an $\oo$-symmetric convex disk. We set
\begin{equation}\label{eq:Cconv_infsup}
\begin{aligned}
\hat{A}_n^{\CC}(\KK) = & \inf \{ \area(\QQ) : \QQ \hbox{ is a } {\CC}-n-\hbox{gon circumscribed about } K \};\\
\hat{a}_n^{\CC}(\KK) = & \sup \{ \area(\QQ) : Q \hbox{ is a } {\CC}-n-\hbox{gon inscribed in } K \};\\
\hat{P}_n^{\CC}(\KK) = & \inf \{ \perim_{\CC}(\QQ) : Q \hbox{ is a } {\CC}-n-\hbox{gon circumscribed about } K \};\\
\hat{p}_n^{\CC}(\KK) = & \sup \{ \perim_{\CC}(\QQ) : Q \hbox{ is a } {\CC}-n-\hbox{gon inscribed in } K \}.
\end{aligned}
\end{equation}
If $\CC= \BB^2$, we omit $\CC$ from the above notation.
\end{Definition}

These quantities, for the special case that $\CC= \BB^2$ and $\KK = r \BB^2$ for some $ < r < 1$ were first investigated in \cite{BLNP}, where the author showed that the best approximating disk-polygons are regular. The case $\CC= \BB^2$ was settled for any spindle convex disk $\KK$ in \cite{FTF2015} by G. Fejes T\'oth and Fodor, who proved the following theorems.

\begin{Theorem}\label{thm:FTF2015_1}
For any spindle convex disk $\KK$ in $\Ee^2$, the sequences $\{ \hat{A}_n(\KK) \}, \hat{P}_n(\KK)$ are convex, and the sequences $\{ \hat{a}_n(\KK) , \hat{p}_n(\KK)\}$ are concave.
\end{Theorem}

\begin{Theorem}\label{thm:FTF2015_2}
Assume that the spindle convex disk $\KK$ has an $m$-fold rotational symmetry, and $n$ is an integer multiple of $m$. Then there is a disk-$n$-gon $\QQ_n$ circumscribed about $\KK$ with $m$-fold rotational symmetry that satisfies $\area(\QQ_n) = \hat{A}_n(\KK)$. Furthermore, the statement remains true if we replace $\hat{A}_n(\KK)$ by $\hat{P}_n(\KK)$, and for inscribed disk-$n$-gons with $\hat{a}_n(\KK)$ and $\hat{p}_n(\KK)$.
\end{Theorem}

We note that it was also shown in \cite{FTF2015} that the above theorems remain true on the sphere and in hyperbolic plane, with the exception of circumscribed perimeter on the sphere; this case is still open in our knowledge. Furthermore, we also note that the linearly convex variant of Theorem~\ref{thm:FTF2015_2} for $k=2$ and area appeared in the original paper of Dowker \cite{D1944}, and was generalized for arbitrary $k$ by L. Fejes T\'oth and G. Fejes T\'oth \cite{FTFT1973}.

Basit and the second-named author investigated the above problem for an arbitrary $\oo$-symmetric convex disk $\CC$ \cite{BL2024}.
Their results are as follows.

\begin{Theorem}\label{thm:BL2024_1}
For any $\oo$-symmetric convex disk $\CC$ in $\Ee^2$ and $\CC$-convex disk $\KK$, the sequences $\{ \hat{A}_n^{\CC}(\KK) \}$, $\{ \hat{P}_n^{\CC}(\KK) \}$ are convex, and the sequence $\{ \hat{p}_n^{\CC}(\KK) \}$ is concave. That is, for any $n \geq 4$, we have
\[
\hat{A}_{n-1}^{\CC}(\KK)+\hat{A}_{n+1}^{\CC}(\KK) \geq 2 \hat{A}_n^{\CC}(\KK), \hat{P}_{n-1}^{\CC}(\KK)+\hat{P}_{n+1}^{\CC}(\KK) \geq 2 \hat{P}_n^{\CC}(\KK), \hbox{ and}
\]
\[
\hat{p}_{n-1}^{\CC}(\KK)+\hat{p}_{n+1}^{\CC}(\KK) \leq 2 \hat{p}_n^{\CC}(\KK).
\]
\end{Theorem}

\begin{Theorem}\label{thm:BL2024_2}
Let $n \geq 3$ and $m \geq 2$. Assume that $n$ is an integer multiple of $m$ and both $\KK$ and $\CC$ have $m$-fold rotational symmetry. Then there are $\CC$-$n$-gons $\QQ^A$, $\QQ^P$ circumscribed about $\KK$ which have $m$-fold rotational symmetry, and $\area(\QQ^A)= \hat{A}_n^{\CC}(\KK)$ and $\perim_{\CC}(\QQ^P)= \hat{P}_n^{\CC}(\KK)$. Similarly, there is a $\CC$-$n$-gon $\QQ^p$ inscribed in $\KK$ which has $m$-fold rotational symmetry, and $\perim_{\CC}(\QQ^p)= \hat{p}_n^{\CC}(\KK)$.
\end{Theorem}

Before our next theorem, we remark that in a topological space $\mathcal{F}$, a subset is called \emph{residual} if it is a countable intersection of sets each of which has dense interior in $\mathcal{F}$. The elements of a residual subset of $\mathcal{F}$ are called \emph{typical}. The next theorem shows that, surprisingly, Dowker's theorem fails for most disks $\CC$.

\begin{Theorem}\label{thm:BL2024_3}
A typical element $\CC$ of the family of $\oo$-symmetric convex disk, with respect to the topology induced by Hausdorff distance, satisfies the property that for every $n \geq 4$, there is a $\CC$-convex disk $\KK$ with
\[
\hat{a}_{n-1}^{\CC}(\KK) + \hat{a}_{n+1}^{\CC}(\KK) > 2 \hat{a}_n^{\CC}(\KK).
\]
\end{Theorem}

The authors of \cite{BL2024} also proved a functional form of Dowker's theorems. Note also that whereas in the special case $\CC= \BB^2$, the sequence $\{ \hat{a}_n^{\CC}(\KK) \}$ is concave for all $\CC$-convex disks $\KK$ by Theorem~\ref{thm:FTF2015_1}, by Theorem~\ref{thm:BL2024_3} there are $\oo$-symmetric convex disks $\CC$ arbitrary close to $\BB^2$ with respect to Hausdorff distance, for which this property fails. Thus, it is a meaningful question to investigate whether this property holds in a neighborhood of $\BB^2$ if we use a finer topology than the one induced by Hausdorff distance. This problem was investigated in the paper \cite{BL2024_2} of Basit and the second named author.

\section{Combinatorial structure of ball polyhedra}\label{sec:combinatorics} 

\emph{Ball-polyhedra}, that is, sets obtained as an intersection of finitely many balls of equal (say, unit) radii in $\Ed$ have been extensively studied, see \cites{BLNP, KMP}, \cite{MMO19}*{Section~8} and the references therein. In this section, we consider ball-polyhedra in $\Ethree$, that have a natural face structure analogous to that of three-dimensional convex polyhedra. Throughout this section $P=\bigcap_{i\in \{1,\dots,N\}}\ball[3]{\vect{p}_i, 1}$ denotes a ball-polyhedron with nonempty interior, $N\geq 3$ such that the centers $\vect{p}_1,\dots,\vect{p}_N$ are all non-redundant, that is, $P\neq\bigcap_{i\in \{1,\dots,N\}\setminus\{j\}}\ball[3]{\vect{p}_i, 1}$ for every $j\in \{1,\dots,N\}$. 

Clearly, $P$ has $N$ \emph{faces:} $P\cap\bd \left(\ball[3]{\vect{p}_i, 1}\right)$ for each $i$. It is not difficult to see that the $i$-th face is a spherically convex set on the  sphere $\bd \left(\ball[3]{\vect{p}_i, 1}\right)$. Edges are slightly more difficult to define, as for every distinct $i$ and $j$, the $i$-th and $j$-th face may intersect in the empty set, in a point, or, in a finite union of points and closed circular arcs, cf. \cite{BN06}. We will call $e$ an \emph{edge} of $P$, if $e$ is a non-degenerate circular arc. The \emph{vertices} of $P$ are easy to define: points of $\bd(P)$ belonging to at least three faces. It is easy to see that the boundary of $P$ is the disjoint union of the spherical relative interior of its faces, the circular relative interior of its edges and the set of vertices. Following \cite{BLNP}, we will say that $P$ is a \emph{standard ball-polyhedron}, if the intersection of any two faces is empty, or one edge, or one vertex. Equivalently, if the faces, edges and vertices form an algebraic lattice with respect to containment.

A fundamental and highly non-trivial result in the combinatorial theory of polytopes is \emph{Steinitz's theorem}, according to which a graph $G$ is the edge graph of a three-dimensional polytope if and only if, $G$ is simple (no parallel edges and loops), planar and 3-connected. Its analogue for ball-polyhedra in $\Ethree$ does not hold as shown by simple examples in  \cite{BN06}, however, it does for standard ball-polyhedra by the main result of \cite{ALN22}.

\subsection{Strongly self-dual convex polytopes and ball-polyhedra}
A special class of convex polytopes in $\Ed$ was introduced by Lov\'asz \cite{Lo83} as follows. We say that $P$ is a \emph{strongly self-dual polytope}, if all vertices of $P$ lie on the unit sphere $\Sedm$, for some $0<r<1$, all facets of $P$ touch the sphere $r\Sedm$, and there is a bijection $\sigma$ between vertices and facets of $P$ such that if $\vect{v}$ is a vertex, then the facet $\sigma(\vect{v})$ is orthogonal to the vector $\vect{v}$. For a vertex $\vect{v}$ of $P$, and a vertex $\vect{u}$ of $\sigma(v)$, we call the line segment $[\vect{u},\vect{v}]$ a \emph{principal diagonal} of $P$. Clearly, every principal diagonal is of length $\alpha=\sqrt{2+2r}$, which we call the \emph{parameter} of $P$. Lov\'asz poses the question of determining the set of possible values of $\alpha$, and gives the following partial answer.

\begin{Theorem}
For every $d\geq 2$ and $\alpha_1<2$, there is a strongly self-dual polytope in $\Ed$ with parameter at least $\alpha_1$.
\end{Theorem}

The motivation to introduce these polytopes was the following. Clearly, the diameter of a strongly self-dual polytope is equal to its parameter $\alpha$. Thus, the diameter graph of the set of vertices of a strongly self-dual polytope $P$ is identical to the graph obtained from the principal diagonals of $P$ as edges of the graph. Proving that  this graph has chromatic number $d+1$, and combining it with combinatorial results, the main result of \cite{Lo83} is that for every $0<\alpha<2$, there is finite set $V$ on $\Sedm$ with the following property: the graph with vertex set $V$ whose edges are the pairs at Euclidean distance $\alpha$, has chromatic number at least $d+1$.

It is well known that a convex set $\body{K}$ in $\Ed$ is a set of constant width (that is, any two parallel supporting hyperplanes are of the same distance) if and only if,
$\body{K}=\bigcap_{x\in\body{K}}\ball{x,r}$, where $r$ is the diameter of $\body{K}$. If $X$ is the vertex set of a regular triangle of side length 1 on the plane, then $\body{K}=\bigcap_{x\in X}\ball[2]{x,1}$, the Reuleaux triangle, is a set of constant width. However, if $X$ is the vertex set of a regular tetrahedron of side length 1 in $\Ethree$, then $\bigcap_{x\in X}\ball[3]{x,1}$ is not a set of constant width, but it may be turned into a set of constant width 1 by ``rounding off the edges'' (to obtain Meissner bodies). Sallee \cite{Sa70} studied this method of constructing sets of constant width, and introduced the term \emph{Reuleaux frame} (called a \emph{Reuleaux polyhedron} in \cite{MPRR20}) referring to ball-polyhedra (intersections of finitely many unit balls), where each center (of the unit balls) is a vertex, and vice versa. Montejano and Rold\'an-Pensado \cite{MR17}, and Martini, Montejano and Oliveros \cite{MMO19} studied further how a Reuleaux polyhedron may be turned into a set of constant width.

Montejano, Pauli, Raggi and Rold\'an-Pensado \cite{MPRR20} considered the following Steinitz-type question for Reuleaux polyhedra. 
Call a planar, 3-connected graph $G$ \emph{strongly self-dual} (cf. \cite{SeSe96}), if there is a self-duality isomorphism $\tau$ between $G$ and its dual $G^{\ast}$ satisfying two conditions.
\begin{itemize}
    \item For every pair of vertices $u, v$, we have that $u \in\tau(v)$ $\Leftrightarrow$ $v \in\tau(u)$, and
    \item for every vertex $u$, we have that $u\notin\tau(u)$.
\end{itemize}
Perhaps currently, the most interesting open problem on the combinatorial structure of three-dimensional ball-polyhedra is the following conjecture from \cite{MPRR20}.

\begin{Conjecture}\label{conj:self-dual}
For every planar, 3-connected, strongly self-dual graph $G$, there is a Reuleaux polyhedron in $\Ethree$ with edge-graph isomorphic to $G$.
\end{Conjecture}

The main result of \cite{MPRR20} is the following weaker statement.
\begin{Theorem}
For every planar, 3-connected, strongly self-dual graph $G$, there is a (not necessarily injective) map $\eta:V\longrightarrow \Ethree$ from the vertex set $V$ such that the diameter of $\eta(V)$ is 1, and for every $v\in V$ and $u\in\tau(v)$ we have $\|\eta(u)-\eta(v)\| = 1$.
\end{Theorem}
We note that the existence of an injective $\eta$ is equivalent to Conjecture~\ref{conj:self-dual}.

\subsection{Diameter graphs}
Borsuk's classical problem asks whether any bounded set in $\Ed$ may be partitioned into $d+1$ subsets, each of smaller diameter than the set. Borsuk confirmed it on the plane, and, to great surprise, it was later proven to be false in high dimensions by Khan and Kalai \cite{KK93}. The answer is still positive for $d=3$ as shown independently by Gr\"unbaum \cite{Gr56}, Heppes \cite{He56} and Straszewicz \cite{Str57} by means of ball polytopes. 
They considered V\'azsonyi's conjecture as well, and proved that in any set of $n$ points in $\Ethree$, the diameter of the set is attained at most $2n-2$ times. For a new proof, see Swanepoel's \cite{Sw08}. The connection between diameter graphs of finite point sets in $\Ed$ an ball-polyhedra was thoroughly explored in \cite{KMP}, including further facts on strongly self-dual polyhedra. 

Continuing the discussion of \emph{diameter graphs}, that is, graphs whose vertex set is a finite set of points in $\Ed$, and a pair form an edge, if they are a diameter of the set, we turn to the following conjecture by Z. Schur, stated in \cite{SchPMK03}: \emph{In any dimension $d$, and for any $n$, the diameter graph of any set of $n$ points in $\Ed$ contains at most $n$ cliques of size $d$.} Note that a size $d$ clique in this graph is the vertex set of a regular $(d-1)$-dimensional simplex. It was confirmed by Hopf and Pannwitz for $d = 2$ in \cite{HP90}, for $d = 3$ by Schur, Perles, Martini and Kupitz in \cite{SchPMK03} and for $d=4$ by Kupavskii \cite{Ku14}. Building on a work of Mori\'c and Pach \cite{MoPa15}, finally, Kupavskii and Polyanskii \cite{KuPo17} confirmed it in full generality. The proof relies on a fine analysis of intersections of balls.

\section{Random ball-polyhedra}\label{sec:random}

\subsection{Intrinsic volumes of random ball-polyhedra}

Extending the work of Rényi and Sulanke \cite{RS1968} from the plane to higher dimensions, Böröczky and Schneider \cite{BSch2010} considered the mean width of random polytopes, that is, convex bodies obtained as the intersection of random half-spaces. Paouris and Pivovarov \cite{PP2017} studied the intrinsic volumes of random ball-polyhedra. They showed the following isoperimetric inequality.
\begin{Theorem}\label{thm:PPintvol}
    Let $N , d \geq 1$ and $R > 0$. Let $f$ be a continuous probability density function on $\Ed$ that is
    bounded by one. Consider independent random vectors $X_1,\dots,X_N$ sampled according to $f$ and 
    $Z_1,\dots,Z_N$ sampled according to the uniform probability measure on $\ball{0,r_d}$,
    where $r_d>0$ is the radius so that $\vol{\ball{0,r_d}}=1$.  Then for all $1 \leq k \leq d$ and $s > 0$, we have
    \[
    \prob{\ivol{\bigcap_{i=1}^N \ball{X_i,R}}>s}\leq
    \prob{\ivol{\bigcap_{i=1}^N \ball{Z_i,R}}>s}.
    \]
\end{Theorem}

As an immediate corollary, we have the following bound on the moments of the intrinsic volumes of random ball-polyhedra.
\begin{Corollary}\label{cor:PPintvol}
 With the assumptions of Theorem~\ref{thm:PPintvol},
    \[
    \Ee\left[\ivol{\bigcap_{i=1}^N \ball{X_i,R}}^p\right]^{1/p}\leq
    \Ee\left[\ivol{\bigcap_{i=1}^N \ball{Z_i,R}}^p\right]^{1/p}
    \]
    holds for any $0<p\leq\infty$.
\end{Corollary}

This result may be interpreted as a stochastic version of the \emph{generalized Urysohn inequality}, according to which
\begin{equation}\label{eq:genUrysohn}
    \left(\frac{\ivol{K}}{\ivol{\ball{o,1}}}\right)^{1/k} \leq \frac{\ivol[1]{K}}{\ivol[1]{\ball{o,1}}}. 
\end{equation}
for any convex body $K$ in $\Ed$ and $1\leq k\leq d$, cf. \cite{Sc14}*{p.603}. To see how Corollary~\ref{cor:PPintvol} yields \eqref{eq:genUrysohn}, consider the body $L=K^R=\bigcap\{\ball{x,R}\st x\in K\}$, the \emph{$R$-dual} of $K$. If $R$ is sufficiently large, then $K$ and the $R$-ball dual $\bigcap\{\ball{x,R}\st x\in L\}$ of $L$ are close to each other. It means that for sufficiently large $R$ and $N$, with high probability, $\bigcap_{i=1}^N \ball{X_i,R}$ is essentially $K$, when the $X_i$ are chosen according to the uniform probability on $L$. Thus, with the normalized characteristic function $f=\frac{1}{\vol{L}}\mathbbm{1}_L$ of $L$ in Corollary~\ref{cor:PPintvol},  inequality \eqref{eq:genUrysohn} follows.

Inequality \eqref{eq:genUrysohn} may be deduced using the Alexandrov--Fenchel inequality (see Schneider's book \cite{Sc14}), or symmetrization techniques. 

    We will use the \emph{Steiner symmetrization}\label{def:SteinerSymm} of a compact set $K$ in $\Ed$ about a $t$-dimensional affine subspace $F$ of $\Ed$, defined as
    \[
    \bigcup\{x+\ball[d-k]{0,r(x)}\st x\in \mathrm{proj}_F(K)\},
    \]
    where for any $x$ in the orthogonal projection of $K$ onto $F$, we denote by $r(x)$ the radius such that 
    $\ivol[d-k]{K\cap(x+F^{\perp})}=\ivol[d-k]{\ball[d-k]{o,r(x)}}$.
    
\noshow{
    The \emph{Minkowski symmetrization} of $K$ about $F$ is defined as
    \[
    \frac{K+R_F(K)}{2},
    \]
    where $R_F$ denotes reflection about $F$. 
}

The proof of Theorem~\ref{thm:PPintvol} relies on a combination of both symmetrization techniques extended to functions, as well as rearrangement inequalities. For another approach, see \cites{AKV2012, FKV2014, FKV, FKV2023}. 

\subsection{An information theoretic approach to the Kneser--Poulsen Conjecture}

Costa and Cover \cite{CC1984} observed deep similarity between the Brunn--Minkowski inequality for the volume of convex bodies and the entropy power inequality for vector valued random variables, which initiated a study of the analogies between convexity theory and information theory. Some fundamental analogies may be summarized in the following table, essentially taken from \cite{AL2024}.
\renewcommand{\arraystretch}{1.5}
\begin{table}[!ht]
    \centering
    \begin{tabular}{|c|c|}
    \hline
    \textbf{Convexity} & \textbf{Information theory}\\
    \hline\hline
    set $K$ in $\Ed$& random vector $X$\\
    \hline
    convex set $K$ in $\Ed$ & log-concave random vector $X$\\
    \hline
    Minkowski sum $K+L$ & sum of independent variables $X + Y$\\
    \hline
    volume $\vol{K}$ & entropy $h_{\alpha}(X)$\\
    \hline
    Euclidean ball $\ball{o,1}$ & \makecell{standard Gaussian vector $Z$,\\ or uniform distribution on $\ball{o,1}$}\\
    \hline
    $\vol{K + r\ball{o,1}}$ & $h_{\alpha}(X + \sqrt{s}Z)$\\    
    \hline
    \end{tabular}\caption{Analogies}\label{tbl:analogies}
\end{table}

As part of this study, Aishwarya, Alam, Li, Myroshnychenko and Zatarain-Vera \cite{AALMZ2023} investigated an information theoretic analogue of the Kneser--Poulsen Conjecture, which, if confirmed in full generality, would yield a proof of the Conjecture. They considered the entropy of sums of random variables.

\begin{Definition}
    Let $X$ be a random vector in $\Ed$ with density function $f$. Then, the \emph{Rényi entropy} of order $\alpha\in(0, 1) \cup (1, \infty)$ of $X$ is defined by
    \[
    h_{\alpha}(X)=\frac{1}{1-\alpha}\log\int_{\Ed} f^{\alpha}(x) \di x. 
    \]
    For $\alpha=0,1,\infty$, it is defined by taking limits. Specifically, denoting the support of $f$ by $\supp(f)=\cl\{x\in\Ed\st f(x)>0\}$, we have
    \begin{equation*}
    h_{0}(X)=\log\vol{\supp(f)}, 
    \end{equation*}
    \begin{equation}\label{eq:SBentropy}
    h_{1}(X)=-\int_{\Ed} f(x)\log f(x) \di x \text{ (called the \emph{Shannon–Boltzmann entropy})} , \text{ and}
    \end{equation}
    \[
    h_{\infty}(X)=-\log\|f\|_{\infty}. 
    \]
\end{Definition}

The Kneser--Poulsen Conjecture for the union may be re-phrased as follows: \emph{For any compact set $K$ in $\Ed$ and contraction $T:\Ed\longrightarrow\Ed$, we have}
\begin{equation}\label{eq:KPunionAsSum}
    \vol{T(K)+\ball{o,1}} \leq \vol{K+\ball{o,1}}.
\end{equation}
 The analogous question in information theory was posed in \cite{AALMZ2023}.
The information theoretic analogue of the ball may be, as in Table~\ref{tbl:analogies}, a Gaussian variable, or more generally, a log-concave variable with radial symmetry. We recall that a random vector in $\Ed$ with density function $f$ is \emph{logarithmically concave} (or log-concave for short), if $f=e^{-\phi}$, where $\phi:\Ed\longrightarrow \Re\cup\{\infty\}$ is a convex function.

\begin{Question}\label{que:entorpy}
Let $X$ and $W$ be independent random vectors in $\Ed$, and  assume that $W$ is log-concave
and satisfies a symmetry property such as radial symmetry. For a
contraction $T$ and $\alpha\in[0,\infty]$, under what additional assumptions do we have
\begin{equation}\label{eq:entorpyQuestion}
h_{\alpha}(T(X)+W)\leq h_{\alpha}(X+W)?  
\end{equation}
\end{Question}

Clearly, an affirmative answer to Question~\ref{que:entorpy} for $\alpha=0$ with $W$ being a uniform random vector from $\ball{o,1}$ would, by \eqref{eq:SBentropy}, yield a proof of the Kneser--Poulsen Conjecture for the union in the form of \eqref{eq:KPunionAsSum}.

On the other hand, an affirmative answer to the Kneser--Poulsen Conjecture for the intersection (cf. \eqref{G-K-W} with equal radii, say $r_i=1$) would yield a proof of Question~\ref{que:entorpy} in some cases. Indeed, set $\alpha=N\geq 2$ an integer, let $W$ again be a uniform random vector from $\ball{o,1}$. Let $X$ (resp., $Y$) be a discrete random vector taking the values $\vect{p}_i$ (resp., $\vect{q}_i$) with probability $1/N$ for $i = 1, \dots, N$. Then $X + W$ has density
\[
\frac{1}{N\cdot \omega_d } \sum_{i=1}^N \mathbbm{1}_{\ball{o,1}}(x - \vect{p}_{i}),
\]
and similarly, $Y+W$ has density
\[
\frac{1}{N\cdot \omega_d } \sum_{i=1}^N \mathbbm{1}_{\ball{o,1}}(x - \vect{q}_{i}).
\]
In this setting, \eqref{eq:entorpyQuestion} is equivalent to
\[
\int_{\Ed} \left( \frac{1}{N\cdot \omega_d } \sum_{i=1}^{N} \mathbbm{1}_{\ball{o,1}}(x - \vect{q}_{i})\right)^N  \di x \geq 
\]\[
\int_{\Ed} \left( \frac{1}{N\cdot \omega_d } \sum_{i=1}^{N} \mathbbm{1}_{\ball{o,1}}(x - \vect{p}_{i})\right)^N  \di x.
\]
If we expand the powers, and use the Kneser--Poulsen conjecture for the volume of the intersection, then a term-by-term comparison yields the inequality above. 

In order to describe the results of \cite{AALMZ2023}, we need to introduce some notions from probability theory. 
For a nonnegative measurable $f: \Ed\longrightarrow [0, \infty)$
which vanishes at infinity, we define its \emph{symmetrically decreasing rearrangement} as the radially symmetric function $f^{\ast}: \Ed\longrightarrow [0, \infty)$ defined (almost everywhere) by the property
\[
\{x\in\Ed\st f^{\ast}(x)>t\} =\ball{o,r_f(x)},
\]
where $r_f(x)$ is the radius for which $\vol{\ball{o,r_f(x)}}=\vol{\{x\in\Ed\st f(x)>t\}}$. In geometric terms, assuming that $f$ is a probability density function, we consider the sub-graph (a set in $\Ee^{d+1}$) of the probability distribution function corresponding to $f$, and take the Steiner symmetrization (see p.~\pageref{def:SteinerSymm} for the definition) of this set about the one dimensional linear subspace spanned by the last basic vector $e_{d+1}$ of $\Ee^{d+1}$. The set thus obtained will be the probability distribution function corresponding to the probability density function $f^{\ast}$. 

For two probability densities $f$ and $g$ on $\Ed$, we say that $f$ is \emph{majorized} by $g$, written
as $f \preceq g$, if 
\[
\int_{\ball{o,r}} f^{\ast} \leq \int_{\ball{o,r}} g^{\ast} 
\]
for all $r>0$. Note that if $f \preceq g$, then $h_{\alpha}(f)\geq h_{\alpha}(g)$ for $\alpha\in[0,\infty]$.

We are ready to state two results of \cite{AALMZ2023}.

\begin{Theorem}
Let $X$ be a log-concave random vector and $W$ be a radially-symmetric log-concave
random vector in $\Ed$, and $\alpha\in(0,\infty)$. Then for any affine contraction $T$, we have
    \begin{equation*}
    f_{X+W} \preceq f_{T(X)+W},    
    \end{equation*}
    and consequently,
    \begin{equation*}
    h_{\alpha}(T(X) + W) \leq h_{\alpha}(X + W).     
    \end{equation*}
\end{Theorem}

\begin{Theorem}
Let $X$ be a random vector in $\Ed$, and $W$ a log-concave random
vector $\Ed$ with radially-symmetric density. 
Then, for any contraction $T:\Ed\longrightarrow\Ed$, any $t>0$, we have
\[
h_2(T(X) + W) \leq h_2(X + W).
\]    
\end{Theorem}

We note that in a recent manuscript \cite{AL2024}, Aishwarya and Li
proved the following.
\begin{Theorem}[The entropic Kneser–Poulsen theorem]\label{thm:entropyKP}
Let $X$ be a random vector, and $Z$ a standard Gaussian
vector $\Ed$, $\alpha\in[0,\infty]$, $s\geq0$. 
Then, for any contraction $T:\Ed\longrightarrow\Ed$, we have
\[
h_{\alpha}(T(X) + \sqrt{s}Z) \leq h_{\alpha}(X + \sqrt{s}Z).
\]  
\end{Theorem}

We cannot resist the temptation to quote from \cite{AL2024} the true information theoretic content of Theorem~\ref{thm:entropyKP}.

Assume that Alice wants to communicate with Bob using the alphabet $\{\vect{p}_1,\dots, \vect{p}_N\}$ (that is, $N$ points in $\Ed$) over a noisy channel 
with additive white Gaussian noise $Z$: Bob receives the random point $\vect{p} + Z$ when Alice sends $\vect{p}$. 
The optimal rate at which information can be reliably transmitted across this noisy
channel cannot be improved by bringing points in $K$ pairwise closer.

\section*{Acknowledgements}

K.B. was partially supported by a Natural Sciences and 
Engineering Research Council of Canada Discovery Grant.

Z.L. was partially supported by the European Research Council Advanced Grant ``ERMiD'', and the NRDI research grant K147544.

M.N. was partially supported by the NRDI research grant K147544 as well as the ELTE TKP 2021-NKTA-62 funding scheme.

\bibliographystyle{amsalpha}
\bibliography{biblio}
\end{document}